\documentclass[3p]{elsarticle}
\usepackage[left]{lineno} 
\usepackage[normalem]{ulem}
\usepackage[hidelinks]{hyperref}
\usepackage{subfigure}
\journal{untitled}
\usepackage{amsmath,amssymb,xcolor}
\usepackage{amsthm}
\newtheorem{theorem}{Theorem}
\newdefinition{rmk}{Remark}
\usepackage{float}
\usepackage{tikz}
\usepackage{url}
\usepackage{booktabs}
\usepackage{epsfig,subfigure,graphicx,float,adjustbox,lipsum}
\floatstyle{plain}\restylefloat{figure}
\floatstyle{plain}\restylefloat{table}
\theoremstyle{definition}

\let\oldproofname=\proofname
\renewcommand{\proofname}{\rm\bf{\oldproofname}}
\usepackage{aliascnt}
\usepackage[nameinlink]{cleveref}
\newaliascnt{lemma}{theorem}

\aliascntresetthe{lemma}
\crefname{lemma}{lemma}{lemmas}

\allowdisplaybreaks

\biboptions{sort&compress}
\begin{document}
	\begin{frontmatter}
		
		\title{\textbf{A Discrete--Time Model of the Academic Pipeline in Mathematical Sciences with Constrained Hiring in the United States
}}
		
		\author[author1]{Oluwatosin Babasola}	
         \corref{mycorrespondingauthor}
		\cortext[mycorrespondingauthor]{Corresponding author}
		\author[author2]{Olayemi Adeyemi }	
        \author[author3]{Ron Buckmire}	
		\author[author4]{Daozhou Gao}	
        \author[author5]{Maila Hallare}
  \author[author6]{Olaniyi Iyiola}
  \author[author7]{Deanna Needell}
  \author[author8,author9,author10]{Chad M. Topaz}
        \author[author11]{Andr\'es R. Vindas-Mel\'endez}

\address[author1]{ University of Georgia, Athens, GA, USA}
\address[author2]{Boise State University, Boise, ID, USA}
\address[author3]  {Marist University,  Poughkeepsie, NY ,  USA}
\address[author4]  {Cleveland State University,  Cleveland , OH ,  USA}
\address[author5]  {United States Air Force Academy, CO ,  USA}
\address[author6]  {Morgan State University,  Baltimore, MD ,  USA}
\address[author7]  {University of California,  Los Angeles, CA ,  USA}
\address[author8]{Williams College, Williamstown, MA, USA}
\address[author9]  {QSIDE Institute, Williamstown , MA,  USA}
\address[author10]  {University of Colorado,  Boulder, CO ,  USA}
\address[author11]  {Harvey Mudd College,  Claremont, CA ,  USA}

\begin{abstract}
The field of the mathematical sciences relies on a continuous academic pipeline in which individuals progress from undergraduate study through graduate training and postdoctoral program to long term faculty employment. National statistics report trends in bachelor's, master's, and doctoral degree awards, but these data alone do not explain how individuals move through the academic system or how structural constraints shape downstream career outcomes. Persistent growth in postdoctoral appointments alongside relatively stable faculty employment indicates that degree production alone is insufficient to characterize workforce dynamics. In this study, we develop a discrete time compartmental model of the academic pipeline in the field of the mathematical sciences that links observed degree flows to latent population stocks. Undergraduate and graduate populations are reconstructed directly from nationally reported degree data, allowing postdoctoral and faculty dynamics to be examined under completion, exit, and hiring processes. Advancement to faculty positions is modeled as vacancy limited, with competition for permanent positions depending on downstream population size. Numerical simulations show that increases in degree inflow do not translate into proportional faculty growth when hiring is constrained by limited turnover. Instead, excess supply accumulates primarily at the postdoctoral stage, leading to sustained congestion and elevated competition. Sensitivity analyses indicate that long run workforce outcomes are governed mainly by faculty exit rates and hiring capacity rather than by degree production alone. These results demonstrate the central role of vacancy limited hiring in shaping academic career trajectories within the field of the mathematical sciences.

\end{abstract}
 \begin{keyword}
 		 Academic pipeline \sep discrete--time model \sep compartmental systems \sep stock–flow dynamics \sep capacity-constrained transitions.
 		\end{keyword}
	\end{frontmatter}

\section{Introduction}
 \noindent The mathematical sciences play a central role in advances in science, engineering, technology, and quantitative decision--making, while also serving as a foundational discipline for academic research and education. Universities support this role through a structured sequence of academic stages consisting of undergraduate education, graduate training, postdoctoral research, and long-term faculty employment. The long-term viability of this system depends on sustained progression through these stages and on alignment between degree production and the capacity of the academic workforce. So, understanding how this progression unfolds over time is critical for assessing the sustainability of the mathematical sciences community. Over years, national datasets have shown substantial variation in the number of bachelor’s, master’s, and doctoral degrees awarded in mathematics and closely related fields, which reflects periods of expansion, stagnation, and renewed growth \cite{NCESdegMath,falkenheim2023academic}. At the same time, persistent concerns remain regarding student attrition, prolonged tenure in temporary academic positions, and limited availability of permanent faculty roles \cite{powell2015future}. These concerns highlight the challenge of inferring downstream workforce dynamics from degree counts alone. Available datasets through the National Center for Education Statistics (NCES) shows that the number of degrees awarded each year, but they do not reveal how individuals progress through the academic system, where attrition occurs, or how structural constraints shape long-term career trajectories. A central problem with this data is that it represent flows rather than populations. Awarded degrees measure annual completions, whereas the academic system consists of evolving stocks of enrolled students, postdoctoral researchers, and faculty members. Consequently, changes in these underlying populations depend not only on degree production but also on completion behavior, exit from academia, and competition for permanent positions.\\
 
 \noindent Therefore, without a framework linking the observed degree flows to these latent populations, it is difficult to assess retention across stages, identify bottlenecks, or determine whether observed patterns reflect the structural features of the academic system. Empirical studies show that attrition occurs throughout the academic pathway \cite{LarsonGhaffarzadeganXue2014,FitzGeraldEtAl2023}. Moreover, undergraduate retention in science and mathematics is consistently lower than institutional averages, with a substantial share of students leaving their initial field of study before completing a degree \cite{correll1997talking,seymour1997talking,chen2013stem,NCESgradRates}. These early losses reduce the pool of students eligible for advanced training and contribute to variability in downstream degree production.\\
 
 \noindent At the graduate level, persistence remains a challenge. Existing studies on doctoral education highlights significant first-year attrition and prolonged time to degree in quantitative fields \cite{golde1998beginning,NSFtimeToDegree}. According to \cite{cleary20102009}, only a subset of doctoral recipients pursue academic careers, and even fewer are able to obtain permanent faculty positions . Following doctoral degree completion, many individuals enter postdoctoral appointments, which function as transitional positions between training and permanent employment. Surveys conducted by the American Mathematical Society (AMS) detail the prevalence of temporary research appointments in the mathematical sciences and show that only a fraction of postdoctoral researchers ultimately secure faculty positions \cite{AMSworkforceReports}. Similar patterns have been observed across the sciences, where postdoctoral employment is characterized by fixed-term contracts, high competition, and uncertain career outcomes \cite{XueLarson2015}.\\
 
 \noindent Studies of scientific labor markets further indicate that prolonged periods in temporary positions are shaped not only by individual productivity but also by structural constraints on hiring and institutional capacity \cite{stephan2015economics}. These structural constraints become most pronounced at the faculty level, which represents the most limited stage of the academic pipeline. Workforce surveys indicate that faculty positions are few in number and characterized by long career durations, with retirements occurring gradually over extended periods \cite{AMSworkforceReports}. This scarcity, combined with prolonged postdoctoral and temporary appointments, suggests that not all doctoral recipients seeking academic careers can secure permanent positions. Moreso, economic analyses of higher education further highlight that hiring decisions are determined by institutional budgets, tenure systems, and long-term employment commitments rather than short-term fluctuations in degree production \cite{ehrenberg2012american}. \\
 
 \noindent A range of quantitative approaches has been applied to education and workforce systems to understand the career progression, and population-level analyses of academic labor markets \cite{larson2014too,ghaffarzadegan2016education}. Population-based models have been used to examine academic career pathways, particularly in biomedical research \cite{ghaffarzadegan2015note, b4} and several work has developed explicit dynamical representations of academic populations. A structured demographic model presented by \cite{MarschkeEtAl2007} shows that faculty composition evolves slowly over time, even when hiring rates or retention policies are substantially altered. Also, analyses based on stochastic update models calibrated to population level data further demonstrate that hiring and attrition influence faculty composition through distinct mechanisms \cite{LaBergeEtAl2024, b2}. These studies show that hypothetical interventions, such as equalizing retention risks or adjusting hiring shares, tend to produce limited short run changes because existing population structure strongly constrains system response \cite{LaBergeEtAl2024}. Furthermore, comparisons between models with uniform transition rules and those with stage specific processes were presented by \cite{ShawStanton2012,ShawEtAl2021} which reveal where losses are concentrated and how these losses differ across demographic groups. Work focused on the postdoctoral workforce further illustrates how modest differences in transition probabilities can accumulate as cohorts advance through the academic system, which leads to persistent disparities over time \cite{GhaffarzadeganHawleyDesai2014}. These results emphasis the importance of modeling transition flows explicitly rather than inferring long--run outcomes from endpoint comparisons alone. However, many existing studies focus on individual stages of the academic pipeline, rely on static snapshots rather than dynamic representations, and only a few explicitly distinguish between observed flows in degree awards and unobserved stocks such as enrolled students, postdoctoral researchers, and faculty populations. As a result, discipline - specific modeling frameworks for the mathematical sciences remain limited, despite the availability of long-term degree and placement data \cite{cleary20102009,topaz2026dynamicalmodelusmathematics}.\\ 
 
 \noindent These limitations motivate the development of a modeling framework that explicitly links observed degree production to underlying academic populations through a mechanistic structure. In this study, the academic pipeline in the mathematical sciences is represented as a discrete--time compartmental system in which individuals progress through undergraduate, graduate, postdoctoral, and faculty stages through the completion, transition, and exit processes. This approach draws on methods commonly used in population biology and epidemiology, where stocks represent populations present at a given time and flows represent transitions between states \cite{b1,brauer2012mathematical,iyi5}. This work is closely related to \cite{topaz2026dynamicalmodelusmathematics}, which develops a discrete--time dynamical model for the U.S. mathematics graduate degree pipeline with time - varying recruitment and completion parameters. The present work complement the study by extending the framework to downstream workforce stages by explicitly representing postdoctoral and faculty populations under vacancy--limited hiring, and state--dependent competition for permanent positions. In particular, \cite{topaz2026dynamicalmodelusmathematics} focuses on explaining observed degree trends, while the present work examines how observed degree flows propagate through a structurally constrained academic workforce. A central feature of the present framework is the explicit separation between observed flows and latent stocks. By introducing observation relationships that link degree awards to underlying populations, undergraduate and graduate stocks are reconstructed directly from empirical data. This formulation allows the dynamic analysis to focus on downstream stages, where structural constraints such as limited hiring capacity and slow faculty turnover are most pronounced.\\ 
 
 \noindent Within this framework, we examine how observed degree flows propagate through the academic pipeline over time, identify conditions under which accumulation arises at intermediate career stages, and explore how long--run workforce structure depends on completion behavior, exit rates, and faculty turnover. By integrating nationally reported degree data with a stock–flow representation of academic progression, this work provides a data--anchored, mechanistic approach for analyzing academic workforce dynamics in the mathematical sciences community. The remainder of this article is structured as follows. Section~\ref{sec2} presents the formulation of the discrete--time pipeline model, parameter definitions, and analytical properties governing system behavior. Section~\ref{sec3} present the results from numerical simulations which emphasize on how degree inflow scenarios and hiring mechanisms influence faculty growth, postdoctoral accumulation, and competition for faculty positions. Section~\ref{sec4} discusses the findings on academic workforce dynamics and structural constraints while Section~\ref{sec5} concludes with a summary of the main contributions and directions for future research.

\section{Method}\label{sec2}

\subsection{Model formulation and stock--flow structure}
 \noindent We represent the academic pathway in the mathematical sciences as a stock--flow system evolving on an annual time scale. The modeling principle adopted is the explicit separation between observable degree outcomes and latent academic populations. The model therefore treats degree awards as flows generated by underlying population stocks rather than as direct measures of workforce size. So, we organize the academic pipeline into four stages: undergraduate students $U_t$, graduate students $G_t$, postdoctoral researchers $P_t$, and faculty members $F_t$. Each stage is regarded as stock, which represents the number of individuals actively present in the corresponding stage during year $t$, while changes in stock size arise from progression from upstream stages and from exits from academia. The conceptual structure of the pipeline is shown in Figure~\ref{fig:pipeline_flow}.

\vspace{10mm}
\begin{figure}[H]
\centering
\resizebox{13cm}{!}{
\begin{tikzpicture}[>=stealth, thick]

\tikzstyle{enter}=[rectangle, rounded corners, draw=black, fill=green!10,
                   minimum width=20mm, minimum height=10mm, align=center]

\tikzstyle{stage}=[rectangle, rounded corners, draw=black, fill=blue!10,
                   minimum width=20mm, minimum height=12mm, align=center]
\tikzstyle{exitbox}=[rectangle, draw=black, fill=red!10,
                  minimum width=20mm, minimum height=10mm, align=center]
\tikzstyle{proc}=[rectangle, rounded corners, draw=black, fill=gray!10,
                  minimum width=20mm, minimum height=12mm, align=center]

\node[enter] (B) at (-4,0) {Inflow\\$B(t)$};
\node[stage] (U) at (0,0) {Undergraduate\\$U_t$};
\node[stage] (G) at (4.5,0) {Graduate\\$G_t$};
\node[stage] (F) at (10.5,0) {Faculty\\$F_t$};

\node[stage] (P) at (4.5,-4.2) {Postdoc\\$P_t$};
\node[proc]  (H) at (8.5,-2.1) {Hiring\\$H_t$};

\node[exitbox] (exitU) at (0,-2.7) {Exit};
\node[exitbox] (exitG) at (4.5,2.7) {Exit};
\node[exitbox] (exitP) at (4.5,-6.9) {Exit};
\node[exitbox] (exitF) at (10.5,-3.7) {Exit};


 \draw[->] (B) -- (U);

\draw[->] (U) -- (G);

\draw[->] (U) -- (exitU);

\draw[->] (G) -- (exitG);

\draw[->] (G) -- (P);

\draw[->] (G) to[bend right=18] (H);

\draw[->] (P) to[bend left=18] (H);

\draw[->] (H) -- (F);


\draw[->] (P) -- (exitP);

\draw[->] (F) -- (exitF);

\end{tikzpicture}
\caption{Academic pipeline flow diagram showing the four-stage stock–flow structure with external inflow, completion-driven transitions, stage-specific exits, and state--dependent progression probabilities.}
\label{fig:pipeline_flow}}
\end{figure}
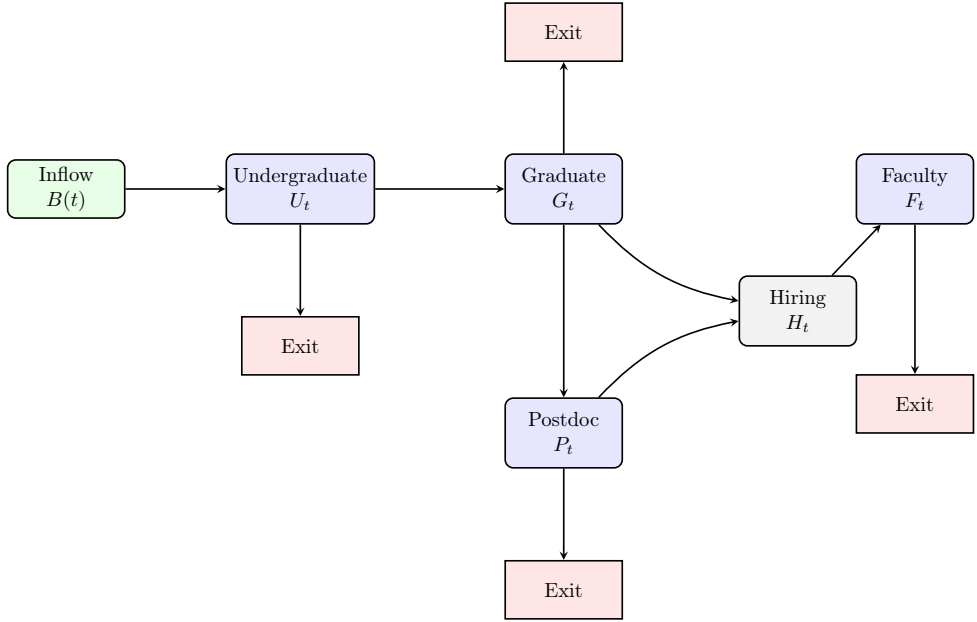
 \noindent As shown in Figure \ref{fig:pipeline_flow}, individuals enter the system through undergraduate study, progress through graduate training and postdoctoral appointments, and may ultimately obtain faculty positions. Exit from academia is permitted at every stage. The undergraduates enter the system through an external inflow and leave the undergraduate stage through degree completion or attrition. Graduate students complete degrees or exit prior to completion. Graduates may transition directly to faculty positions, enter postdoctoral appointments, or exit academia. Postdoctoral researchers occupy temporary academic positions and may either transition to faculty roles or exit the academic system. Faculty members leave the system through retirement or permanent departure, creating vacancies for new hires. The temporal evolution of the system is governed by the following recurrence relations:
\begin{equation}
\begin{aligned}
U_{t+1} &= (1 - g^U - a^U)U_t + B(t), \\
G_{t+1} &= (1 - g^G - a^G)G_t + p^{UG}(G_t) g^U U_t, \\
P_{t+1} &= \left(1 - p^{PF}(F_t) - a^P\right)P_t + p^{GP} g^G G_t, \\
F_{t+1} &= (1 - a^F)F_t + p^{GF} g^G G_t + p^{PF}(F_t)P_t. 
\end{aligned}\label{model}
\end{equation}
 The parameters $g^U$ and $g^G$ denote the annual probabilities of degree completion at the undergraduate and graduate stages, respectively, while $a^U$, $a^G$, $a^P$, and $a^F$ represent the annual probabilities of exiting the academic pipeline from each corresponding stage. The term $B(t)$ denotes the external inflow of new students entering the undergraduate population in year $t$. Progression between stages is governed by state--dependent transition probabilities. The function $p^{UG}(G_t)$ represents the probability that an undergraduate completer enters graduate study and depends on the size of the graduate population $G_t$, reflecting limits on graduate program capacity. Similarly, $p^{PF}(F_t)$ denotes the probability that a postdoctoral researcher transitions to a faculty position and depends on the current faculty population $F_t$, capturing competition for a fixed number of permanent academic positions. The system in \eqref{model} represents an unconstrained hiring regime in which all qualified candidates can transition into faculty positions according to the modeled transition probabilities.\\
 
 \noindent To study the effects of limited hiring capacity, we also consider a vacancy--limited regime in which faculty hires are capped by available openings generated through expansion or faculty exits. The vacancy--limited mechanism modifies the effective transitions to faculty while preserving the same candidate supply structure. Although the full model is presented as a four-stage academic pipeline, the undergraduate and graduate populations are reconstructed directly from observed degree data through  the observation equations. As a result, $U_t$ and $G_t$ are treated as data--anchored inputs rather than endogenously forecast state variables. The dynamic core of the model therefore governs the evolution of postdoctoral and faculty populations, which respond to the reconstructed upstream flows under vacancy--limited hiring and competition mechanisms.

\subsection{Vacancy limited hiring mechanism}
\noindent Faculty hiring is explicitly constrained by the availability of open positions. We have $a^F$ which denotes the annual faculty exit probability. The number of faculty vacancies $(V_t)$ created in year $t$ is
\begin{equation}
V_t = a^F F_t .
\end{equation}
Candidates for faculty positions arise from two sources. Graduate completers may transition directly into faculty roles, generating a candidate supply $(C_t^{dir})$, expressed as
\begin{equation}
C_t^{dir} = p^{GF}\,g^G\,G_t .
\end{equation}
Postdoctoral researchers form a second candidate pool, with supply $C_t^{post}$ which is defined as
\begin{equation}\label{eq4}
C_t^{post} = p^{PF}(F_t)\,P_t ,
\end{equation}
where $p^{PF}(F_t)$ is a state--dependent transition intensity. Total faculty hires in year $t$ are limited by the number of available vacancies and this is given as
\begin{equation}\label{eq5}
H_t = \min \left\{ V_t,\; C_t^{dir} + C_t^{post} \right\}.
\end{equation}
The quantity $H_t$ denotes total faculty hires and we have also have $H_t^{post}$ which denotes the number of postdoctoral researchers hired into faculty positions in year $t$. When both candidate pools are nonzero, hires are allocated proportionally to candidate supply. The proportional allocation assumes that hiring committees draw from candidate pools in proportion to their relative availability. This is a simplifying modeling assumption intended to separate candidate supply from realized hires. In practice, hiring preferences may favor particular candidate categories, such as postdoctoral researchers or recent graduates. Under this assumption, the number of postdoctoral hires is given by
\begin{equation}\label{eq6}
H_t^{post} = H_t \frac{C_t^{post}}{C_t^{dir} + C_t^{post}} .
\end{equation}
Competition for faculty positions is represented by a decreasing postdoctoral transition intensity,
\begin{equation}
p^{PF}(F_t) = \frac{p^{PF}_{\max}}{1 + \alpha^F F_t / K_F} ,
\end{equation}
where $K_F$ sets the scale of the faculty workforce and $\alpha^F$ controls the strength of competition. The functional form for the postdoctoral to faculty transition probability is not intended as an empirically estimated hazard. Rather, it is a phenomenological choice designed to enforce monotonicity, boundedness, and interpretable limiting behavior. As faculty numbers increase relative to capacity, the effective hiring probability decreases smoothly and remains within admissible bounds. The parameters $\alpha_F$ and $K_F$ act as scaling quantities that determine the strength and onset of competition, rather than calibrated estimates tied to observed hiring data. This formulation ensures bounded transitions and prevents unbounded faculty growth. Capacity constraints in graduate education are represented by a state--dependent undergraduate to graduate transition probability,
\begin{equation}
p^{UG}(G_t) = p^{UG}_{\max} \left( 1 - \frac{G_t}{K_G} \right)_+ ,
\end{equation}
where $K_G$ denotes the graduate capacity scale and $(\cdot)_+$ indicates truncation at zero.

\subsection{Vacancy--limited state updates}
\noindent The quantities defined above describe candidate supply and realized hiring under vacancy--limited conditions. In this regime, candidate flows do not translate directly into state transitions as in the unconstrained model. Instead, faculty hires are capped by available vacancies and allocated proportionally between candidate pools as described in  \eqref{eq5} -- \eqref{eq6}. As a result, the population dynamics for postdoctoral researchers and faculty differ from the unconstrained recurrences in \eqref{model}. Under vacancy--limited hiring, faculty and postdoctoral populations evolve according to
\begin{equation}
\begin{aligned}
F_{t+1}& = (1 - a^F)\,F_t + H_t,\\
P_{t+1}& = P_t - H_t^{\text{post}} - a^P P_t + p^{GP} g^G G_t,
\end{aligned}
\end{equation}
 where $H_t$ remains the total faculty hires and $H_t^{\text{post}}$ denotes the subset of hires drawn from the postdoctoral pool. We note that the undergraduate and graduate populations are reconstructed directly from observed degree data and are therefore treated as exogenous inputs in this regime.

\subsection{Data and model calibration}
\noindent This section describes the empirical data used to inform the model, the manner in which these data are linked to the latent model states, and the procedure used to assign baseline parameter values. The objective is to explicitly discuss how observed quantities constrain the model while clearly distinguishing between directly observed data, latent populations inferred and parameters specified through empirical evidence or structural assumptions. It is important to note that the focus is not to estimate or forecast academic workforce sizes from data, but rather to examine how observed degree flows propagate through a structurally constrained academic pipeline. The model parameters are externally calibrated from published sources while some are chosen to reflect existing career durations and transition patterns, rather than estimated by fitting to workforce time series.

\subsubsection{Observed degree data}
\noindent The primary empirical inputs consist of national annual degree counts in the mathematical sciences. Specifically, we use time series of bachelor’s, master’s, and doctoral degree awards compiled from publicly available national datasets ranging across $1970 - 2020$. These data provide consistent annual measurements of degree production but do not directly report enrollment or workforce populations defined under a common classification across time.\\
Suppose $B_s(t)$ denote the number of bachelor’s degrees awarded in year $t$, and let $M_s(t)$ and $D_s(t)$ denote the numbers of master’s and doctoral degrees awarded, respectively. These series are treated as reliable observations of annual completion flows at the undergraduate and graduate stages. In contrast, consistent national time series for undergraduate enrollment, graduate enrollment, postdoctoral employment, and faculty employment are not available across the full historical period. These quantities are therefore treated as latent states within the modeling framework.

\subsubsection{Linking observed degrees to latent stocks}
 \noindent The model distinguishes explicitly between latent population stocks and observable annual outcomes. The state variables $U_t$, $G_t$, $P_t$, and $F_t$ represent the numbers of individuals actively present in undergraduate study, graduate study, postdoctoral appointments, and faculty positions during year $t$. Observed degree counts are interpreted as flows generated by completion processes acting on the corresponding enrollment stocks. Observed degree production is linked to the model through the observation equations,
\begin{equation}
B_s(t) = g^U U_t,
\qquad
M_s(t) = r_{M}\, g^G G_t,
\qquad
D_s(t) = r_D\, g^G G_t,
\label{eq:obslink}
\end{equation}
where $g^U$ and $g^G$ remains the annual completion probabilities at the undergraduate and graduate stages, and $r_{M}$ and $r_D$ denote the proportions of graduate completions receiving master’s and doctoral degrees, respectively, with $r_{M}+r_D=1$. The degree composition parameters are estimated directly from the observed data as
\begin{equation}
r_{M} = \frac{\sum_t M_s(t)}{\sum_t \bigl(M_s(t)+D_s(t)\bigr)},
\qquad
r_D = 1-r_{M}.
\label{eq:rM}
\end{equation}
\noindent Using the available national degree series, we obtain $r_{M}\approx0.80$ and $r_D\approx0.20$, consistent with long--run national patterns in graduate degree production. Because enrollment stocks defined consistently with the degree series are not directly observed, undergraduate and graduate populations are reconstructed from the observation equations, such that
\begin{equation}
U_t = \frac{B_s(t)}{g^U},
\qquad
G_t = \frac{M_s(t)+D_s(t)}{g^G}.
\label{eq:reconstructUG}
\end{equation}
\noindent This reconstruction anchors the upstream stages of the model directly to empirical degree production and avoids attempting to estimate highly correlated combinations of enrollment levels and completion probabilities from degree flows alone. In this formulation, the observed degree series act as exogenous inputs that drive downstream workforce dynamics. Postdoctoral and faculty populations are not reconstructed directly from data. Instead, they are generated endogenously by the model through graduate placement, vacancy--limited hiring, and exit processes. 
\subsubsection{Parameter calibration}
 \noindent Model parameters are assigned using a combination of empirical anchoring, calibration to aggregate statistics, and structural specification. This approach reflects the heterogeneous nature of the available data. Completion probabilities at the undergraduate and graduate stages, $g^U$ and $g^G$, are anchored to published statistics on degree completion and typical time to degree. These parameters are treated as externally informed quantities rather than estimated from the degree time series, since the reconstruction in \eqref{eq:reconstructUG} already enforces consistency between completion probabilities and observed degree flows. For the exit probabilities, undergraduates and graduates exit probabilities, $a^U$ and $a^G$, are informed by reported retention and attrition statistics in higher education datasets. The postdoctoral exit probability $a^P$ reflects the finite duration of postdoctoral appointments and transitions into non--academic careers and is calibrated to reported postdoctoral tenure distributions. The faculty exit probability $a^F$ is informed by faculty age distributions and retirement patterns reported in surveys. Placement fractions for graduate completers are specified using aggregate placement statistics. The parameter $p^{GP}$ and $p^{GF}$ are constrained to satisfy $p^{GP}+p^{GF}\leq 1$, with the remaining fraction exiting academia. Because placement outcomes are typically reported as multi--year averages rather than annual time series, these parameters are treated as time invariant in the baseline model.\\
 
 \noindent Structural parameters governing capacity--constrained transitions are not directly observable and are therefore specified based on system scale and explored through sensitivity analysis. The graduate capacity parameter $K_G$ controls saturation in the undergraduate-to-graduate transition, while the faculty capacity parameter $K_F$ controls suppression of postdoctoral hiring into faculty positions as the faculty workforce grows. The competition parameter $\alpha^F$ governs the strength of this suppression. Maximum transition intensities $p_{\max}^{UG}$ and $p_{\max}^{PF}$ define upper bounds on progression in the absence of capacity constraints and are chosen to be consistent with historically observed placement rates. All baseline parameter values, interpretations, and data sources are summarized in Table~\ref{tab:baseline-params}.
\vspace{5mm}
\begin{table}[H]
  \caption{Baseline model parameters, interpretations, and data sources.}
\label{tab:baseline-params}
		\resizebox{16cm}{!}{
			\begin{tabular}{clcl}
				\toprule
				Parameter & Description & Baseline value & Source \\
				\midrule

\hline
$B(t)$ 
& Annual inflow of new undergraduate entrants 
& Observed 
& Reconstructed from NCES degree data \cite{NCESdegMath} \\

$g^U$ 
& Undergraduate degree completion probability 
& $0.14$ 
& Estimated from six-year STEM completion rates \cite{NCESgradRates,PCASTSTEM} \\

$a^U$ 
& Undergraduate exit probability without completion 
& $0.12$ 
& Estimated from STEM attrition studies \cite{seymour1997talking,chen2013stem} \\

$p^{UG}$ 
& Fraction of undergraduate completers entering graduate study 
& State dependent 
& Assumed with capacity limitation, informed by AMS and NSF trends \\

$K_G$ 
& Graduate program capacity scale 
& $25{,}000$ 
& Assumed scale based on national graduate enrollment magnitudes \\

$g^G$ 
& Graduate degree completion probability 
& $0.17$ 
& Estimated from median time to degree statistics \cite{NSFtimeToDegree} \\

$a^G$ 
& Graduate exit probability without completion 
& $0.08$ 
& Estimated from doctoral attrition literature \cite{golde1998beginning} \\

$p^{GP}$ 
& Fraction of graduate completers entering postdoctoral positions 
& $0.45$ 
& Estimated from AMS placement outcomes \cite{cleary20102009} \\

$p^{GF}$ 
& Fraction of graduate completers entering faculty positions 
& $0.08$ 
& Estimated from AMS placement outcomes \cite{cleary20102009} \\

$a^P$ 
& Postdoctoral exit probability 
& $0.25$ 
& Estimated from typical postdoctoral duration \cite{powell2015future} \\

$p^{PF}_{\max}$ 
& Maximum postdoctoral transition intensity to faculty 
& $0.18$ 
& Estimated from longitudinal placement summaries \cite{AMSworkforceReports} \\

$a^F$ 
& Faculty exit probability 
& $0.03$ 
& Estimated from average faculty career length \cite{AMSworkforceReports} \\

$K_F$ 
& Faculty workforce capacity scale 
& $4{,}000$ 
& Assumed order of magnitude from national faculty counts \cite{AMSworkforceReports} \\

$\alpha^F$ 
& Strength of competition for faculty positions 
& $1.0$ 
& Assumed neutral competition strength \\

\bottomrule   
				 \end{tabular}
            }
		
    \end{table}
 \noindent We further note that the parameters governing competition and capacity are not estimated from data. They are introduced to capture structural constraints on hiring and to prevent unrealistically unbounded growth. Also, the choice $\alpha^F=1.0$ serves as a neutral baseline scaling; alternative values are explored through sensitivity analysis.

\subsubsection{Consistency with observed degree trends}
 \noindent As an internal consistency check, we compare the degree flows implied by the reconstructed enrollment stocks with the observed national degree series. Figure~\ref{count} displays the observed master’s, and doctoral degree counts alongside the corresponding model-generated outputs obtained from the observation equations. It show the consistency check between observed degree data and model-implied degree outputs under the reconstruction framework. By construction, the model reproduces the observed degree trends, confirming that the reconstruction and calibration procedures are internally consistent. These comparisons are not interpreted as measures of predictive performance but serve to verify the linkage between observed data and latent model states. The calibrated parameters and reconstructed initial conditions are used in all subsequent simulations.

\begin{figure}[H]
\centering

\subfigure[]{%
\label{Grad}
\includegraphics[width=0.5\textwidth, height = 5.5cm]{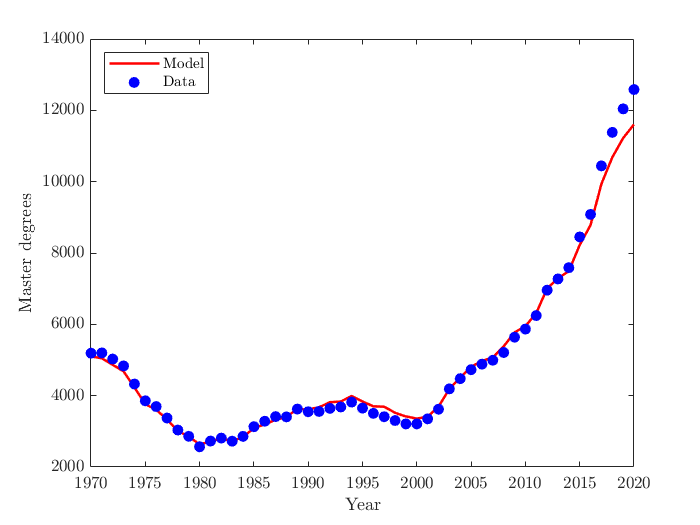}
}%
\subfigure[]{%
\label{Doct}
\includegraphics[width=0.5\textwidth, height =5.5cm]{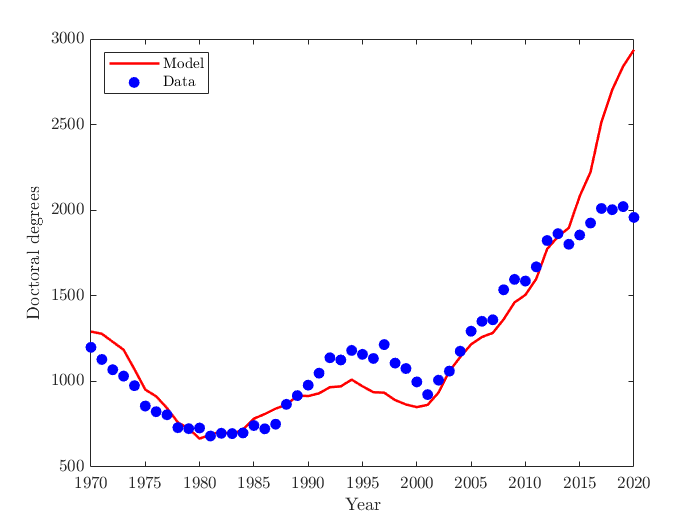}
}
\caption{Observed national degree counts (blue points ) and model-implied degree flows (red lines) for master’s and doctoral degrees. Model outputs are computed from reconstructed graduate and doctoral populations using the observation equations.}
\label{count}
\end{figure}
 \noindent The agreement between observed and model--generated degree series in Figure~\ref{count} is interpreted as an internal consistency check rather than an out of sample validation. Because undergraduate and graduate populations are reconstructed directly from observed degree data at each time step, this agreement is largely guaranteed by construction and does not constitute empirical validation of downstream workforce dynamics.

\subsection{Sensitivity analysis}
 \noindent To evaluate the robustness of the modeled workforce dynamics and to identify the structural mechanisms that most strongly regulate postdoctoral congestion and faculty outcomes, we conducted both local parameter perturbation analysis and global sensitivity analysis around the baseline parameterization. The analysis focuses on parameters governing progression, exit, hiring capacity, and competitive pressure, and examines their influence on both terminal and peak values of faculty and postdoctoral populations. The objective is to distinguish parameters that primarily redistribute individuals within the pipeline from those that fundamentally alter long run workforce structure. We first examined the response of model outcomes to independent perturbations of individual parameters while holding all others fixed at baseline values. This local analysis evaluates how marginal changes in a single structural mechanism propagate through the pipeline. Figures~\ref{fig:sensitivity_oat1} to \ref{fig:sensitivity_oat3} summarize the effects of these perturbations on final and maximum faculty and postdoctoral population levels while Figure \ref{fig:sensitivity_oat4} shows the global sensitivity.

\begin{figure}[H]
\centering
\subfigure[]{%
\label{fig:sensitivity_oat_alphaF}
\includegraphics[width=0.33\textwidth,height=5.6cm]{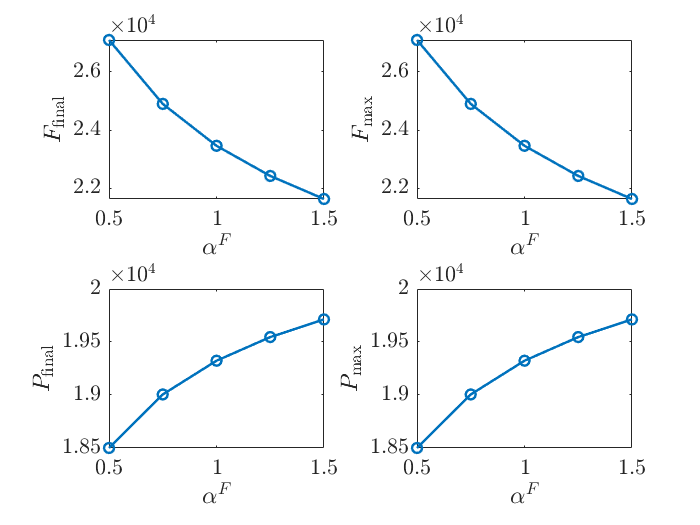}
}%
\subfigure[]{%
\label{fig:sensitivity_oat_aP}
\includegraphics[width=0.34\textwidth,height=5.6cm]{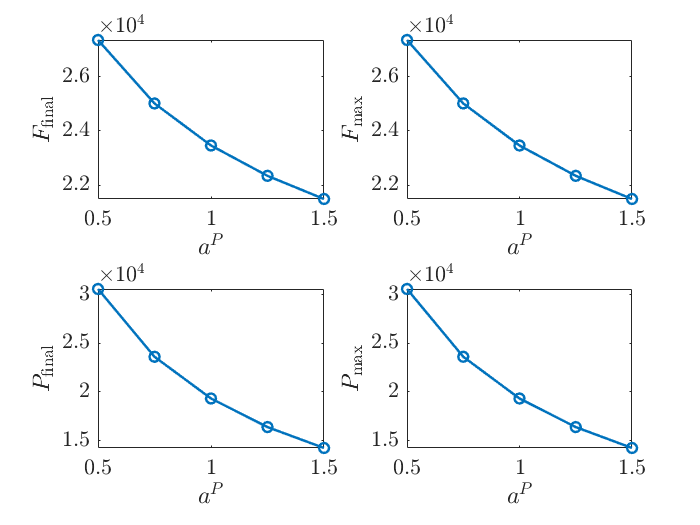}
}%
\subfigure[]{%
\label{fig:sensitivity_oat_aF}
\includegraphics[width=0.33\textwidth,height=5.6cm]{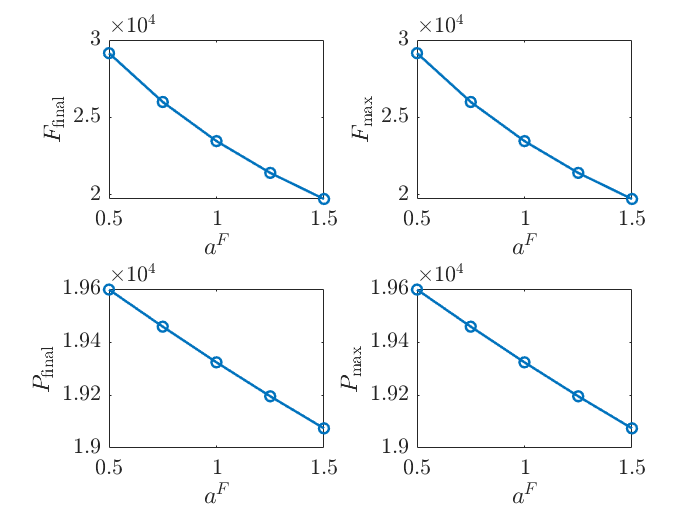}
}
\caption{Local parameter perturbation analysis showing sensitivity of final and peak faculty and postdoctoral populations to (a) variation in competition intensity, (b) postdoctoral exit, and (c) faculty exit parameters.}
\label{fig:sensitivity_oat1}
\end{figure}

\noindent  Figure~\ref{fig:sensitivity_oat_alphaF} shows sensitivity to the competition intensity parameter governing postdoc-to-faculty transitions. It is observed that increasing competitive pressure reduces both final and peak faculty populations while simultaneously increasing postdoctoral accumulation. This pattern indicates that congestion is shaped not only by the number of candidates entering the hiring pool but also by how competition suppresses the efficiency with which postdoctoral researchers convert into permanent positions. We illustrates the effect of postdoctoral exit rates in Figure~\ref{fig:sensitivity_oat_aP}. From this figure, we observed that higher postdoctoral exit reduces both terminal and peak postdoctoral populations, but also leads to lower faculty outcomes. This reflects the role of postdoctoral exits as a loss mechanism for potential faculty candidates. Figure~\ref{fig:sensitivity_oat_aF} examines sensitivity with respect to the faculty exit rate. This shows that increasing faculty exit lowers final and peak faculty size directly, but also reduces postdoctoral accumulation by increasing the availability of hiring opportunities. This result highlights faculty turnover as a central regulator of both workforce stability and congestion. Next, we show the influence of hiring capacity and faculty transition in Figure~\ref{fig:sensitivity_oat2}.

\begin{figure}[H]
\centering
\subfigure[]{%
\label{fig:sensitivity_oat_pPFmax}
\includegraphics[width=0.5\textwidth,height=5.6cm]{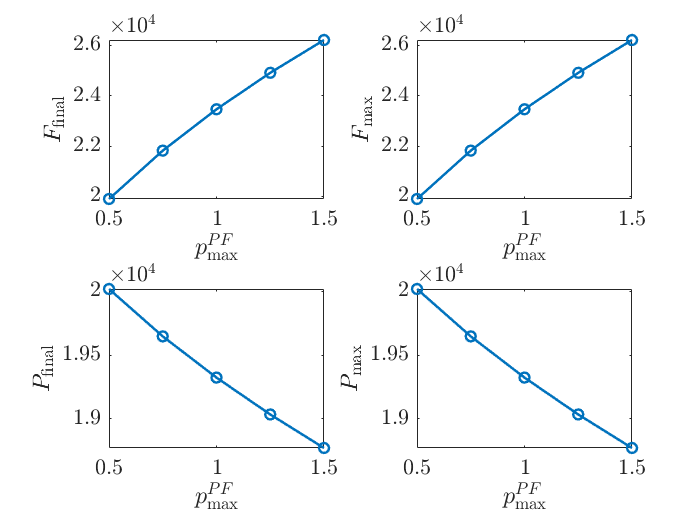}
}%
\subfigure[]{%
\label{fig:sensitivity_oat_KF}
\includegraphics[width=0.5\textwidth,height=5.6cm]{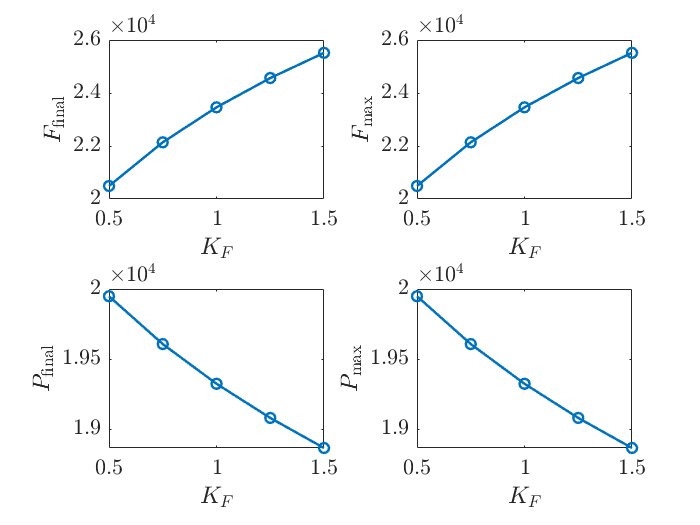}
}
\caption{Sensitivity of workforce outcomes to hiring throughput parameters, including the maximum postdoc to faculty transition rate and the faculty capacity scale.}
\label{fig:sensitivity_oat2}
\end{figure}

\noindent The influence of postdoc to faculty transition and hiring capacity is shown in Figure~\ref{fig:sensitivity_oat_pPFmax} and \ref{fig:sensitivity_oat_KF} respectively.
Across both, the approximately linear monotone response indicates that changes in hiring throughput propagate proportionally to downstream workforce stocks, with no evidence of saturation or threshold behavior within the explored parameter range. The figures show similar trend and it suggests that increasing either the maximum postdoc to faculty transition probability or the faculty capacity scale leads to higher faculty populations and reduced postdoctoral accumulation. The results demonstrate that sustained reductions in congestion require structural expansion of hiring capacity rather than adjustments to upstream degree flows alone.
The absence of nonlinear response suggests that marginal changes to transition probabilities are insufficient to relieve accumulation when faculty hiring capacity remains fixed. Also, we examine the sensitivity of the graduate progression to both faculty and postdoctoral position.
This is shown in Figure~\ref{fig:sensitivity_oat3}.

\begin{figure}[H]
\centering
\subfigure[]{%
\label{fig:sensitivity_oat_pGF}
\includegraphics[width=0.5\textwidth,height=6.5cm]{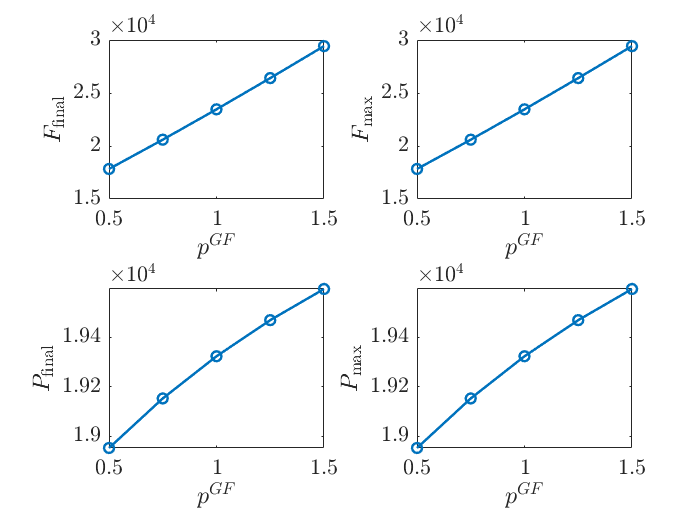}
}%
\subfigure[]{%
\label{fig:sensitivity_oat_pGP}
\includegraphics[width=0.5\textwidth,height=6.5cm]{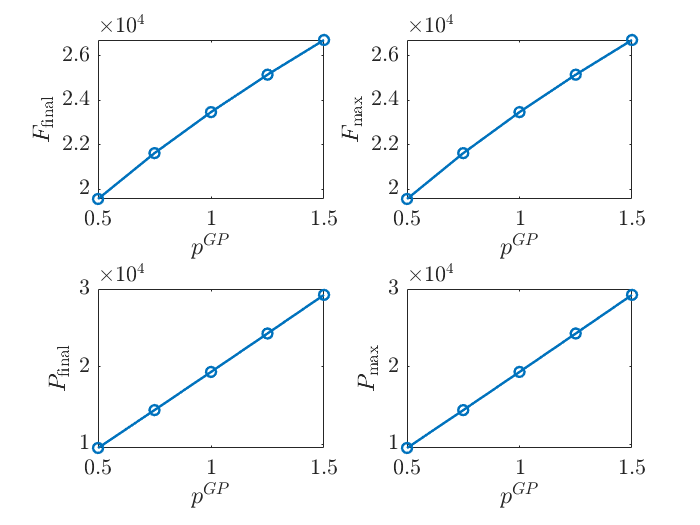}
}
\caption{Sensitivity of workforce outcomes to graduate progression pathways, (a) direct transitions to faculty and (b) transitions into postdoctoral positions.}
\label{fig:sensitivity_oat3}
\end{figure}

\noindent Figures~\ref{fig:sensitivity_oat_pGF} and \ref{fig:sensitivity_oat_pGP} examine sensitivity to graduate progression probabilities. In both cases, the near linear sensitivity indicates that upstream progression choices primarily redistribute congestion between compartments rather than eliminating it. These show that an increase in direct transitions from graduate study to faculty raises faculty populations and moderately increases postdoctoral levels, while increasing transitions into postdoctoral positions produces substantial growth in postdoctoral accumulation. These results suggest that upstream pathway choices influence where congestion manifests but do not eliminate congestion when downstream hiring remains constrained.\\

\noindent To complement the local perturbation analysis, we conducted a global sensitivity analysis using partial rank correlation coefficients. Figure~\ref{fig:sensitivity_oat4} summarizes the correlations between model parameters and final faculty size across a broad region of parameter space. Faculty outcomes are most strongly and positively associated with hiring throughput and capacity parameters, and negatively associated with faculty exit and competition intensity. Postdoctoral exit also exhibits a strong negative association, reflecting its role in reducing the pool of candidates available for faculty replenishment. The smooth and monotonic structure of these correlations indicates that the qualitative behavior of the model does not rely on finely tuned parameter values. Instead, long--run workforce outcomes are governed primarily by structural parameters controlling faculty turnover, hiring capacity, and competitive pressure. While progression parameters influence the size of the candidate pool, congestion emerges robustly whenever hiring capacity grows more slowly than candidate supply. This confirms that postdoctoral accumulation is an inherent feature of vacancy--constrained academic systems and persists across wide regions of parameter space.
\begin{figure}[H]
\centering
\includegraphics[width=0.9\textwidth,height=9cm]{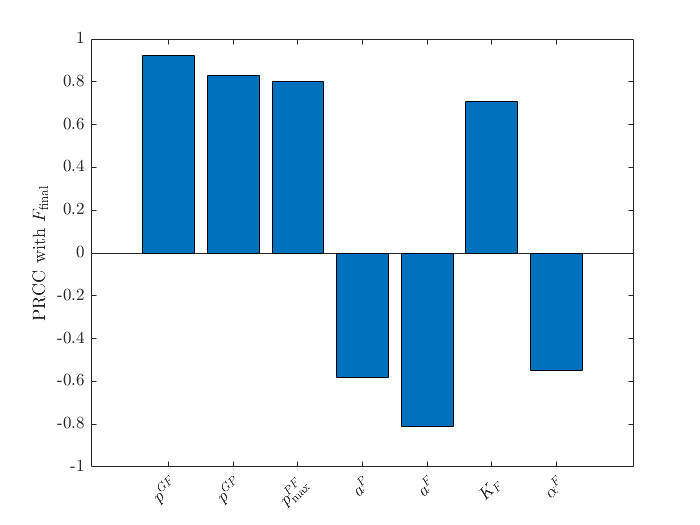}
\caption{Global sensitivity analysis using partial rank correlation coefficients showing the relative influence of model parameters on final faculty population size.}
\label{fig:sensitivity_oat4}
\end{figure}
 \noindent The sensitivity analysis focuses on parameters governing downstream transitions, exits, and hiring constraints. Completion probabilities at the undergraduate and graduate levels are not included in the analysis because they enter the model through algebraic reconstruction of upstream stocks rather than through dynamic transition equations. In particular, undergraduate and graduate populations are reconstructed as $U_t = B_S(t)/g^U$ and $G_t = (M_S(t)+D_S(t))/g^G$, so variation in $g^U$ and $g^G$ primarily rescales candidate supply rather than altering the structural mechanisms driving congestion or hiring outcomes. Consequently, uncertainty in these parameters affects magnitudes but not the qualitative dynamics identified here. In addition, the analysis focuses on monotonic effects of individual parameters and does not explicitly capture higher--order interactions. While partial rank correlation coefficients identify dominant single--parameter influences, variance--based approaches such as Sobol indices would be required to quantify interaction effects among parameters.

\subsection{Model analysis}
 \noindent We analyze the discrete--time academic pipeline model to establish that it is mathematically well--posed as a population accounting system. In particular, we show that model trajectories remain within a physically meaningful feasible region and that all state variables remain bounded under bounded inflow and positive exit. We also characterize conditions under which accumulation occurs at the postdoctoral and faculty stages. We define the feasible region
\[
\Omega = \{(U,G,P,F)\in\mathbb{R}^4 : U\geq 0,\; G\geq 0,\; P\geq 0,\; F\geq 0\},
\]
which represents nonnegative population states. Feasibility of the model means that trajectories start and remain in \(\Omega\) for all future time steps.

\begin{theorem}[Positivity and feasibility]\label{thm:positivity}
Assume \(U_0,G_0,P_0,F_0 \geq 0\) and \(B(t)\geq 0\) for all \(t\geq 0\). Suppose the coefficients satisfy
\[
g^U+a^U \leq 1,\qquad g^G+a^G \leq 1,\qquad p^{PF}(F_t)+a^P \leq 1 \quad \text{for all } t,
\]
and that the transition functions \(p^{UG}(G_t)\) and \(p^{PF}(F_t)\) are nonnegative for all \(t\).
Then the solution satisfies
\[
(U_t,G_t,P_t,F_t)\in\Omega \quad \text{for all } t\geq 0.
\]
\end{theorem}

\begin{proof}
We show that the update map sends \(\Omega\) into itself and conclude by induction.
Assume \((U_t,G_t,P_t,F_t)\in\Omega\)

\begin{enumerate}[(i)]
\item Since \(g^U+a^U\leq 1\) and \(B(t)\geq 0\),
\[
U_{t+1}=(1-g^U-a^U)U_t + B(t) \geq 0.
\]

\item Since \(g^G+a^G\leq 1\) and \(p^{UG}(G_t), g^U, U_t\geq 0\),
\[
G_{t+1}=(1-g^G-a^G)G_t + p^{UG}(G_t)g^U U_t \geq 0.
\]

\item Since \(p^{PF}(F_t)+a^P\leq 1\) and \(p^{GP},g^G,G_t\geq 0\),
\[
P_{t+1}=\bigl(1-p^{PF}(F_t)-a^P\bigr)P_t + p^{GP}g^G G_t \geq 0.
\]

\item Since \(a^F\leq 1\) and all inflow terms are nonnegative,
\[
F_{t+1}=(1-a^F)F_t + p^{GF}g^G G_t + p^{PF}(F_t)P_t \geq 0.
\]
\end{enumerate}

\noindent Thus the map preserves \(\Omega\). Since the initial condition lies in \(\Omega\), induction implies feasibility for all \(t\geq 0\).
\end{proof}

\begin{theorem}[Boundedness under bounded inflow and positive exit]\label{thm:boundedness}
Assume \(0\leq B(t)\leq B_{\max}\) for all \(t\), and assume
\[
g^U+a^U>0,\qquad g^G+a^G>0,\qquad a^P>0,\qquad a^F>0.
\]
Under the constraints of Theorem~\ref{thm:positivity}, all trajectories \((U_t,G_t,P_t,F_t)\) are bounded for all \(t\geq 0\).
\end{theorem}

\begin{proof}
We derive comparison inequalities for each stock.

\noindent For undergraduates, we have
\[
U_{t+1} \leq (1-g^U-a^U)U_t + B_{\max}.
\]
Let \(c_U = 1-(g^U+a^U)\). Since \(g^U+a^U>0\), we have \(0\leq c_U<1\). Iteration yields
\[
U_t \leq U_0 + \frac{B_{\max}}{g^U+a^U},
\]
so \(U_t\) is bounded.

\noindent For graduates, we have
\[
G_{t+1} \leq (1-g^G-a^G)G_t + g^U U_t.
\]
Let \(c_G=1-(g^G+a^G)\) with \(0\leq c_G<1\). Since \(U_t\) is bounded, standard comparison implies boundedness of \(G_t\).

\noindent For postdocs, we have
\[
P_{t+1} \leq (1-a^P)P_t + p^{GP}g^G G_t.
\]
Since \(a^P>0\), let \(c_P=1-a^P\in[0,1)\). Boundedness of \(G_t\) implies boundedness of \(P_t\).

\noindent For faculty, we have
\[
F_{t+1} \leq (1-a^F)F_t + p^{GF}g^G G_t + P_t.
\]
With \(c_F=1-a^F\in[0,1)\) and bounded \(G_t,P_t\), it follows that \(F_t\) is bounded. Thus, all state variables are bounded.
\end{proof}

\begin{theorem}[One-step accumulation conditions]\label{thm:accumulation}
Define increments \(\Delta P_t=P_{t+1}-P_t\) and \(\Delta F_t=F_{t+1}-F_t\). Then
\[
\Delta P_t>0 \iff p^{GP}g^G G_t > \bigl(p^{PF}(F_t)+a^P\bigr)P_t,
\]
and
\[
\Delta F_t>0 \iff p^{GF}g^G G_t + p^{PF}(F_t)P_t > a^F F_t.
\]
\end{theorem}

\begin{proof}
Subtracting \(P_t\) from the postdoctoral update equation gives
\[
\Delta P_t = -\bigl(p^{PF}(F_t)+a^P\bigr)P_t + p^{GP}g^G G_t,
\]
which yields the stated condition.  
Similarly, subtracting \(F_t\) from the faculty update equation gives
\[
\Delta F_t = -a^F F_t + p^{GF}g^G G_t + p^{PF}(F_t)P_t,
\]
which yields the stated condition.
\end{proof}

\section{Results}\label{sec3}
 \noindent Throughout this section, simulation results are interpreted as conditional workforce trajectories driven by observed degree flows and the specified pipeline structure, rather than as forecasts of realized undergraduate or graduate enrollments. We investigate the dynamic behavior of postdoctoral and faculty populations under different hiring regimes using numerical simulations of the model. The analysis focuses on how observed degree inflows propagate through the academic pipeline, how competition for faculty positions emerges under vacancy constraints, and which structural parameters govern long--run postdoctoral congestion. We begin by examining faculty dynamics under unconstrained hiring, then  the competition intensity and accumulation under vacancy--limited hiring. Results are presented in Figures~\ref{fac}--\ref{map}.
\vspace{10mm}
 \begin{figure}[H]
\centering
\subfigure[]{%
\label{fig:fig01_faculty_unconstrained_index}
\includegraphics[width=0.32\textwidth,  height = 5cm]{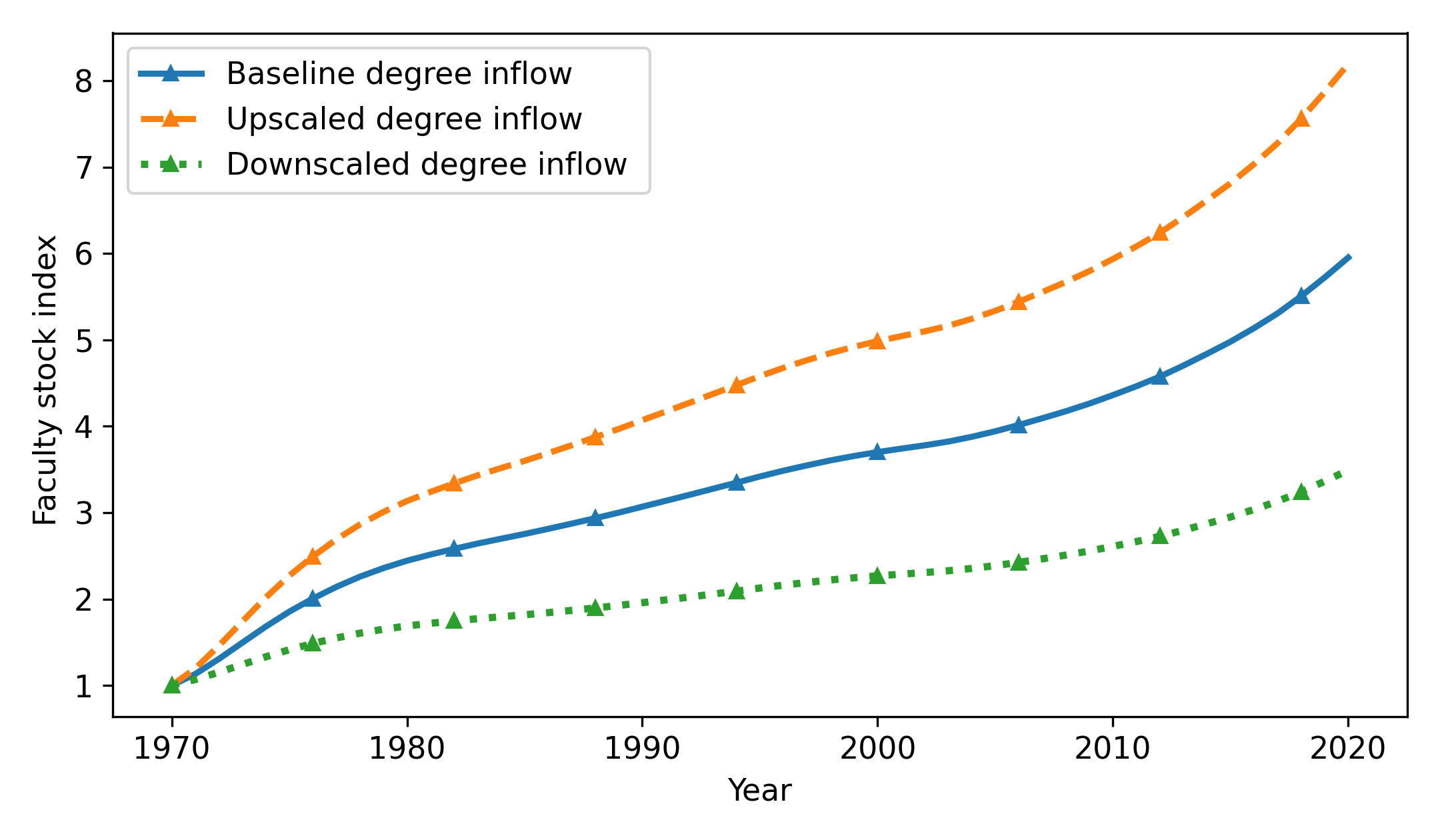}
}%
\subfigure[]{%
\label{fig:fig02_market_pressure_unconstrained}
\includegraphics[width=0.32\textwidth, height = 5cm]{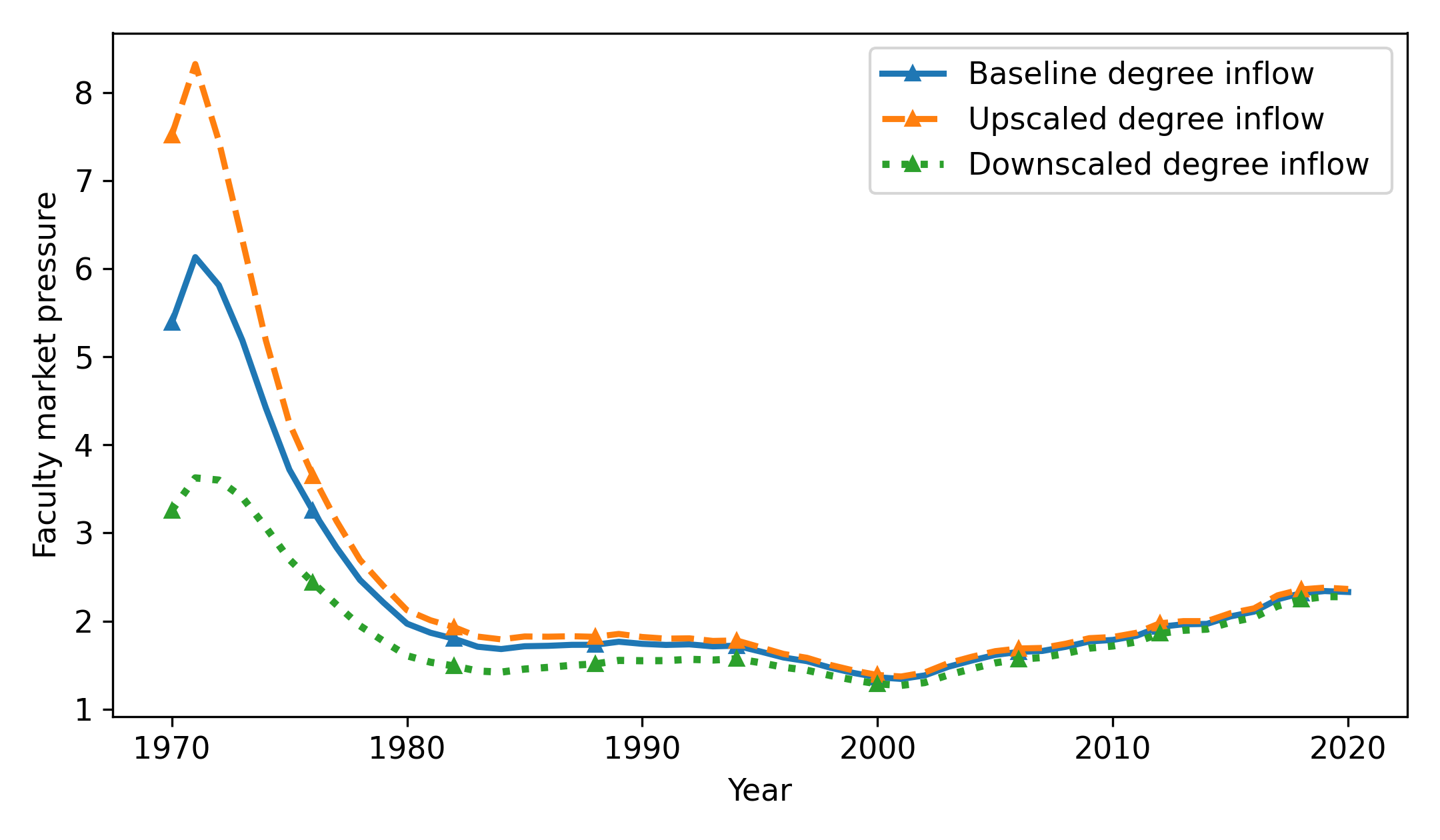}
}%
\subfigure[]{%
\label{fig:fig04_competition_intensity_vacancy_limited}
\includegraphics[width=0.32\textwidth, height =5cm]{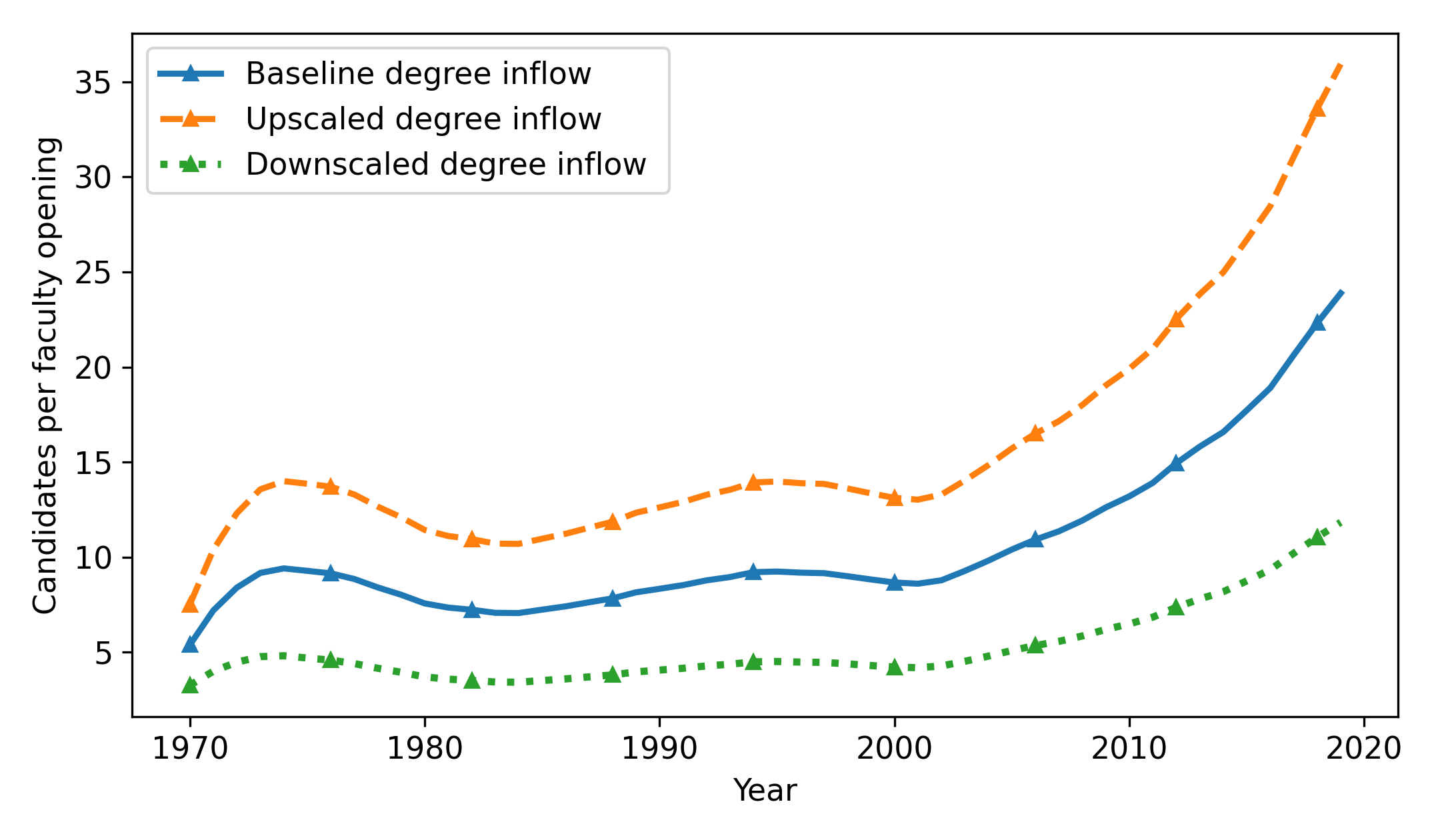}
}
\caption{Faculty trajectories, market pressure, and competition intensity under alternative degree inflow scenarios. (a) faculty trajectories under unconstrained dynamics, where hiring is governed solely by modeled transition rates rather than vacancy availability.  (b) faculty market pressure under the same unconstrained setting, defined as the ratio of potential faculty candidates to openings generated by faculty exits.  (c) competition intensity for faculty positions under vacancy limited hiring, defined as candidates per available opening. }\label{fac}
\end{figure}
\noindent Figures~\ref{fac} presents faculty dynamics under unconstrained and vacancy-limited hiring regimes, illustrating how degree inflow influences faculty growth and competition intensity under differing structural constraints. Figure~\ref{fig:fig01_faculty_unconstrained_index} shows the faculty trajectories across degree inflow scenarios when hiring is governed solely by modeled transition rates and not restricted by vacancy availability. This setting isolates the influence of upstream degree inflow on faculty growth in the absence of institutional hiring bottlenecks. The figure shows that faculty stocks increase under expanded degree inflow and grow more slowly under reduced inflow, with differences that widen progressively over time. This behavior indicates that when hiring capacity adjusts freely, faculty size remains sensitive to changes in upstream supply, and long run faculty levels can diverge substantially across degree inflow scenarios. \\

\noindent Figure~\ref{fig:fig02_market_pressure_unconstrained} quantifies faculty market pressure under the same unconstrained regime, where pressure is defined as the ratio of potential faculty candidates to openings generated by faculty exits. The purpose of this figure is to assess whether changes in degree inflow alter the balance between candidate supply and replacement demand even when hiring is not vacancy limited. The figure shows that market pressure increases when degree inflow is scaled upward and decreases when inflow is reduced. This indicates that upstream expansion generates candidate pools that grow faster than exit driven openings, leading to sustained competition even in the absence of explicit hiring constraints. The result highlights that slow faculty turnover alone can generate excess competition, independent of institutional hiring limits. Figure~\ref{fig:fig04_competition_intensity_vacancy_limited} shifts attention to the vacancy--limited regime and shows competition intensity for faculty positions when hiring is restricted by available openings. Competition intensity is defined as the number of candidates per faculty opening. In contrast to the unconstrained case, the figure shows that competition responds more strongly to changes in degree inflow under vacancy--limited hiring. Expanded inflow produces a pronounced rise in competition intensity, while reduced inflow dampens competition throughout the time horizon. This comparison demonstrates that vacancy constraints amplify the effects of upstream degree changes on competition, and transforming supply increases into congestion rather than faculty expansion. We next examine how vacancy--limited hiring reshapes postdoctoral dynamics and hiring outcomes. Postdoctoral trajectories across degree inflow scenarios, the composition of faculty hires, and postdoctoral projections are shown in Figure \ref{postdoc}.
\\\\

 \begin{figure}[H]
\centering
\subfigure[]{%
\label{fig:fig03_postdoc_vacancy_limited_index}
\includegraphics[width=0.32\textwidth,  height = 5cm]{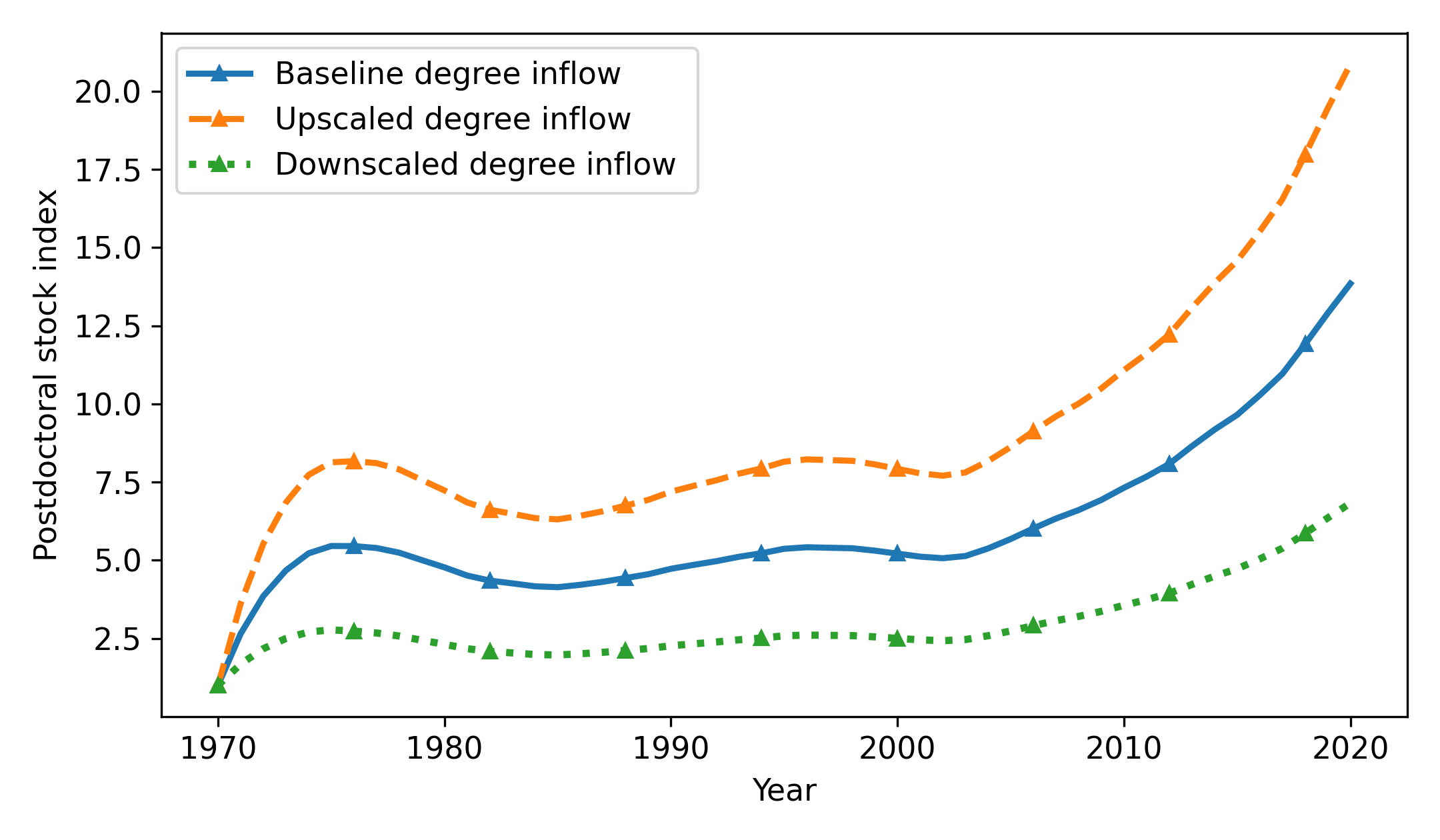}
}%
\subfigure[]{%
\label{fig:fig05_hire_composition_postdoc_share}
\includegraphics[width=0.32\textwidth, height = 5cm]{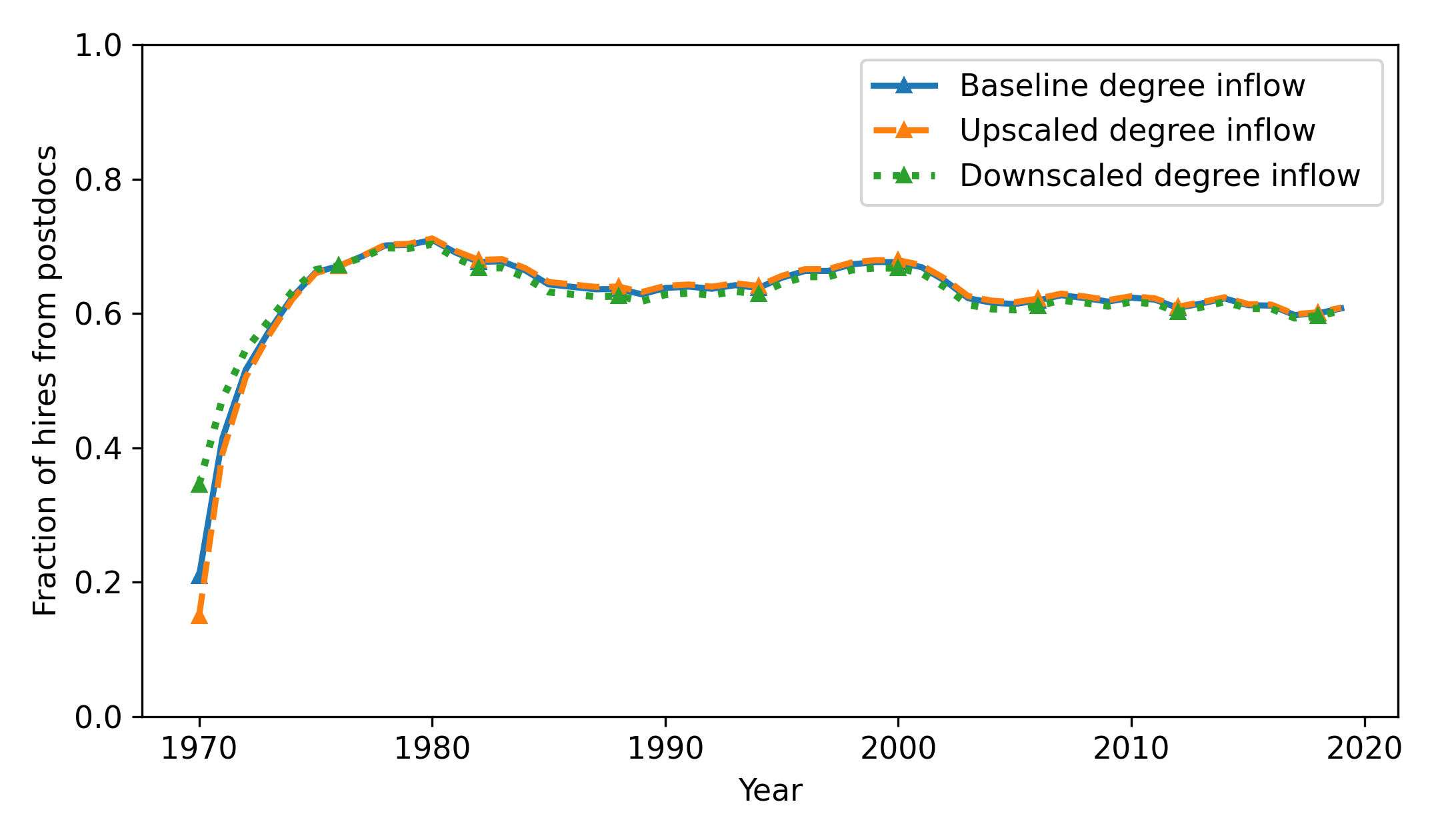}
}%
\subfigure[]{%
\label{fig:fig06_postdoc_projection_index}
\includegraphics[width=0.32\textwidth, height =5cm]{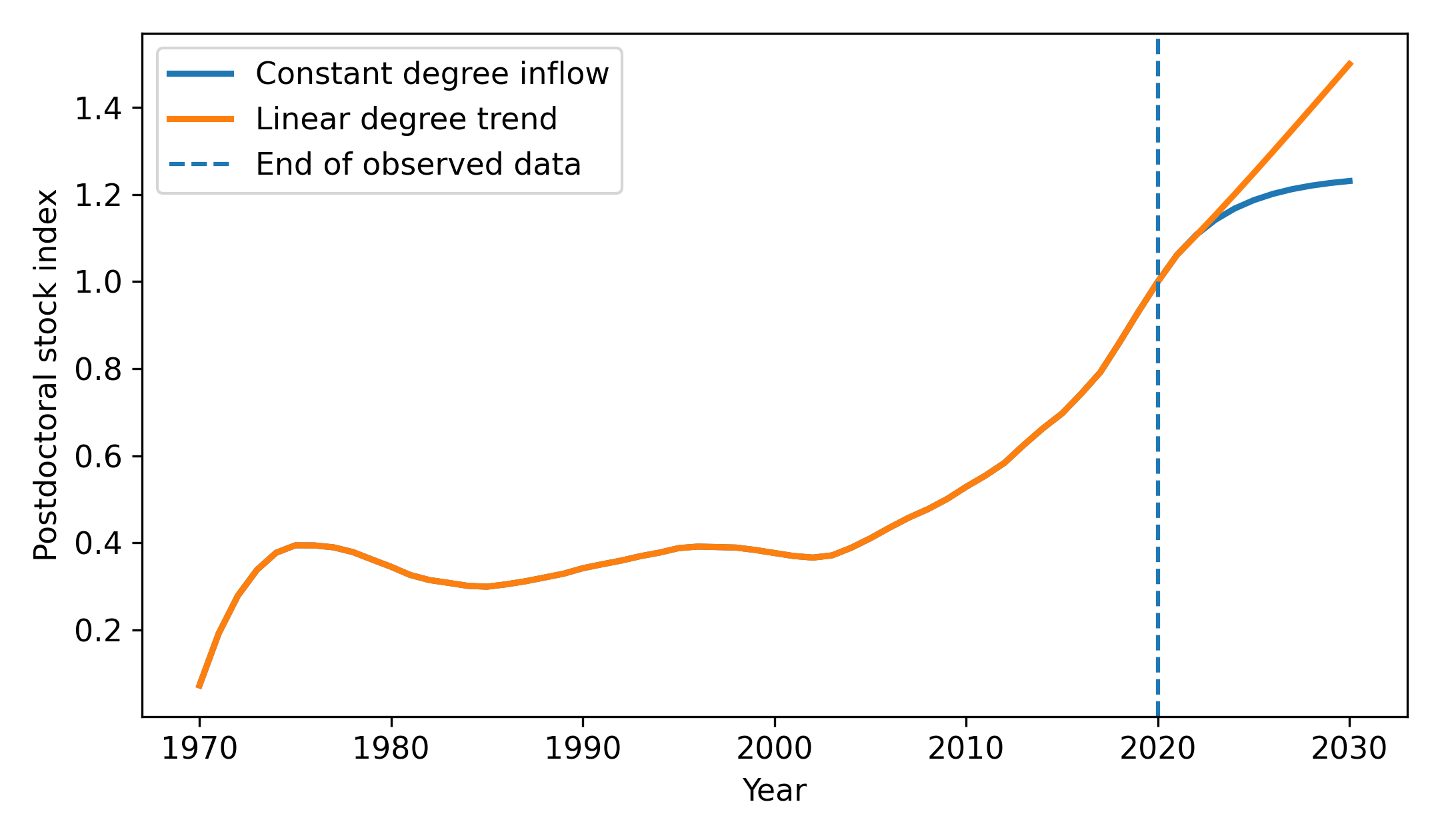}
}
\caption{Postdoctoral trajectories, hiring composition, and postdoctoral projections under vacancy limited hiring. (a) postdoctoral stock trajectories across degree inflow scenarios. (b) fraction of faculty hires drawn from the postdoctoral pool. (c) postdoctoral projections under degree inflow. }\label{postdoc}
\end{figure}
\noindent Figure \ref{postdoc} presents postdoctoral outcomes under vacancy-limited hiring and show how variation in degree inflow affects postdoctoral accumulation and hiring outcomes. Figure \ref{fig:fig05_hire_composition_postdoc_share} examined whether increased degree inflow leads to postdoctoral accumulation when hiring capacity is fixed. It indicates that expanded degree inflow results in substantial postdoctoral accumulation relative to baseline, while reduced inflow suppresses postdoctoral growth. This pattern suggests that the postdoctoral stage functions as a buffer that absorbs excess supply when downstream hiring is constrained, preventing proportional growth in faculty positions. Furthermore, to determine the overall competition intensity from changes in hiring composition, we examines the composition of faculty hires by displaying the fraction of hires drawn from the postdoctoral pool under vacancy limited hiring as shown in  Figure~\ref{fig:fig05_hire_composition_postdoc_share}.  This shows that despite large differences in degree inflow and postdoctoral congestion, the share of hires originating from postdoctoral positions remains relatively stable over time which indicates that increased degree inflow intensifies competition without substantially altering the balance between direct and postdoctoral hiring channels.\\

\noindent Next, we assess how strongly projected postdoctoral growth depends on upstream degree assumptions under vacancy limited hiring. This is done by exploring the sensitivity of postdoctoral projections to assumptions about degree inflow beyond the observed data window and the result is shown in  Figure~\ref{fig:fig06_postdoc_projection_index} which shows that long run postdoctoral trajectories diverge markedly depending on whether degree inflow stabilizes or continues to increase, even though hiring capacity remains unchanged. This output demonstrates that uncertainty in future postdoctoral stocks is driven primarily by degree inflow dynamics rather than by hiring rules alone. Finally, we examine how structural workforce parameters shape long--run congestion outcomes. The terminal postdoc to faculty ratio, its geometric structure, and the timing of congestion are shown in Figure \ref{map}.
\begin{figure}[H]
\centering
\subfigure[]{%
\label{fig:fig07_heatmap_terminal_ratio}
\includegraphics[width=0.33\textwidth,  height = 5cm]{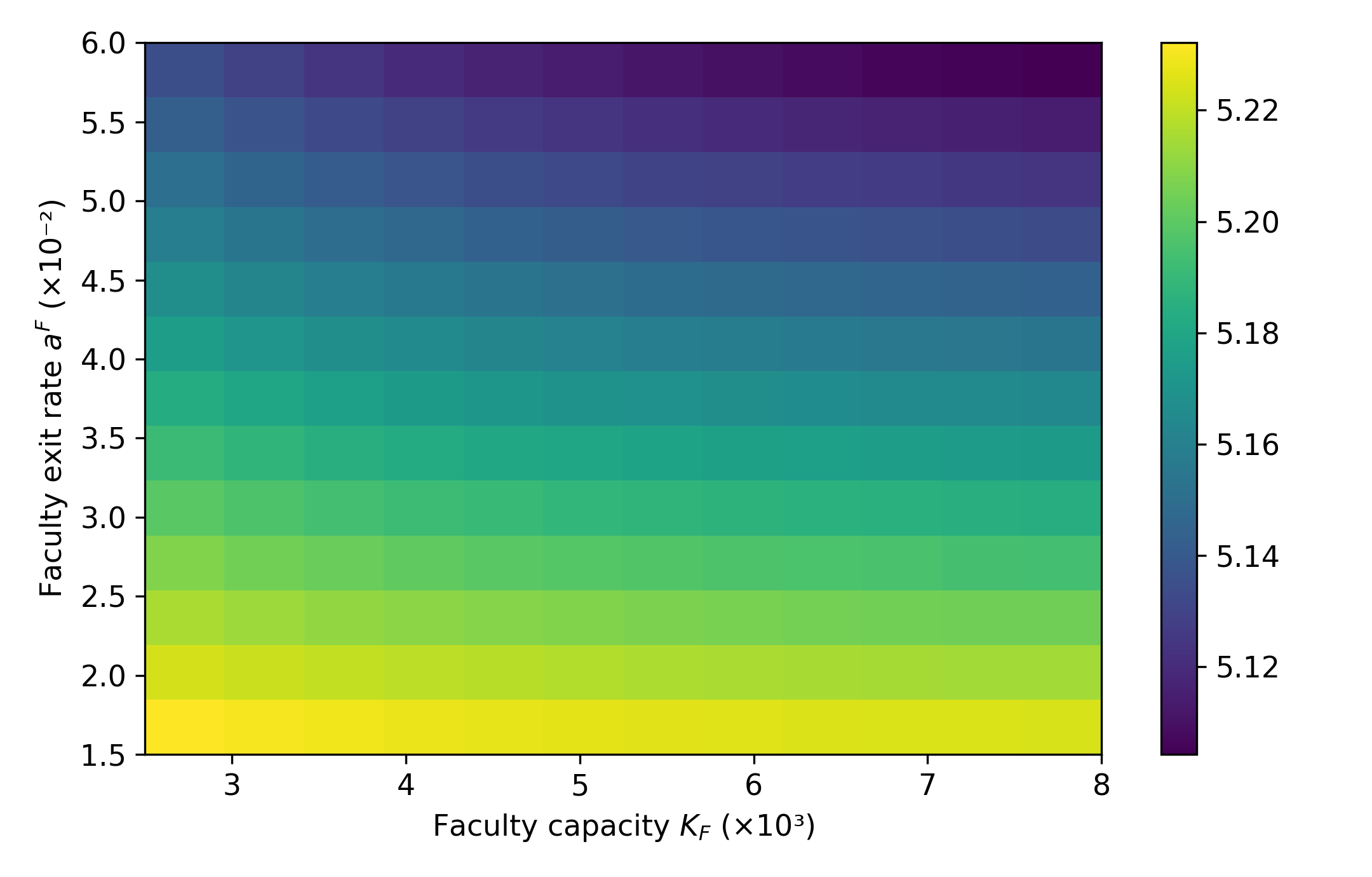}
}%
\subfigure[]{%
\label{fig:fig08_contour3d_terminal_ratio}
\includegraphics[width=0.34\textwidth, height = 5cm]{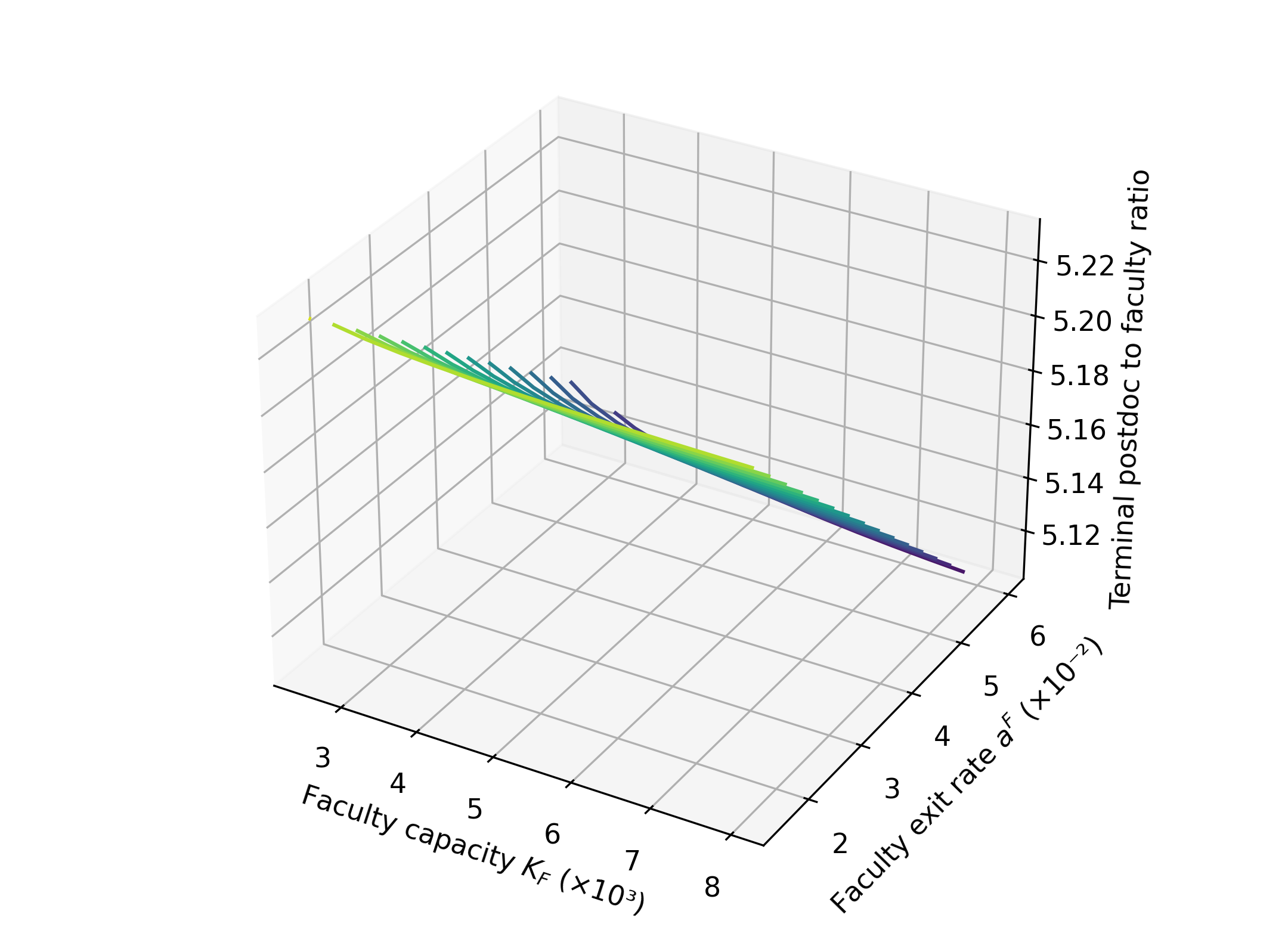}
}%
\subfigure[]{%
\label{fig:fig09_heatmap_first_year_threshold}
\includegraphics[width=0.33\textwidth, height =5cm]{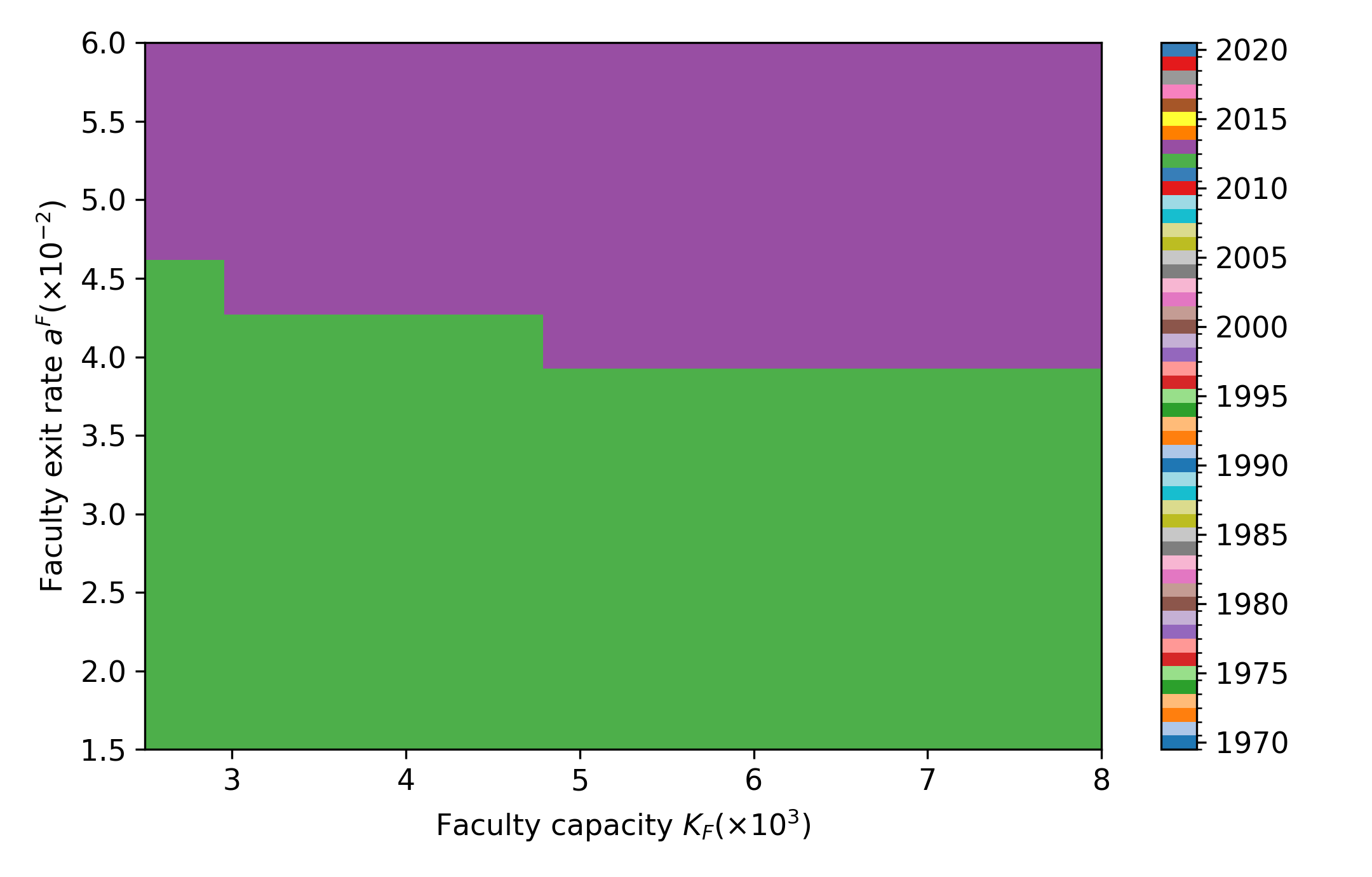}
}
\caption{Long--run postdoctoral congestion as a function of faculty exit rates and hiring capacity. (a)  heat map of the terminal postdoc to faculty ratio across combinations of faculty exit rate and capacity scale.  (b)  contour representation of the  terminal congestion surface, (c) first year in which the postdoc to faculty ratio exceeds a  congestion threshold. }\label{map}
\end{figure}
 \noindent Figure~\ref{fig:fig07_heatmap_terminal_ratio} presents a heat map of the terminal postdoc to faculty ratio as a function of the faculty exit rate and the faculty capacity scale. The purpose of this figure is to identify which structural parameters govern long run congestion at the postdoctoral stage. The figure shows systematic variation in the terminal ratio across parameter combinations, with lower exit rates and smaller capacity scales associated with higher postdoctoral burden relative to faculty. This suggests that slow faculty turnover and limited hiring capacity jointly reinforce congestion, independent of short-term degree fluctuations. \\
 
 \noindent Figure~\ref{fig:fig08_contour3d_terminal_ratio} provides a three dimensional contour representation of the same terminal congestion surface. The purpose of this visualization is to clarify the joint dependence of congestion on exit rates and capacity scales. The smooth contour structure indicates that similar congestion outcomes arise along continuous combinations of structural parameters, reinforcing the interpretation that postdoctoral congestion is an emergent system property rather than the result of isolated parameter values. Figure~\ref{fig:fig09_heatmap_first_year_threshold} shows the first year in which the postdoc to faculty ratio exceeds a specified congestion threshold, using a discrete color scale to emphasize timing differences. The purpose of this figure is to examine when congestion emerges under varying exit and capacity settings. The figure shows that under low exit rates and tight capacity, congestion appears earlier within the historical window, whereas higher exit rates or larger capacity scales delay its onset. This result highlights that both the magnitude and timing of postdoctoral congestion are governed by structural features of faculty turnover and hiring capacity. In general, these results demonstrate that expanding degree inflow does not translate into proportional faculty growth when hiring is constrained by vacancies. Instead, excess supply accumulates in the postdoctoral stage, intensifying competition while leaving hiring composition largely unchanged. Long--run congestion and its timing are controlled primarily by faculty exit rates and hiring capacity scales, indicating that structural workforce parameters play a dominant role in shaping academic career dynamics.

\section{Discussion}\label{sec4}
 \noindent The results demonstrate that downstream academic workforce outcomes in the mathematical sciences are governed primarily by structural constraints rather than by degree production alone. Accordingly, the contribution of this work lies in clarifying structural mechanisms that generate postdoctoral accumulation and constrained faculty growth, rather than in statistically estimating or validating workforce trajectories from data.
 Across all simulations, postdoctoral accumulation arises whenever faculty hiring is limited by turnover, even when degree production is held constant or reduced. This indicates that congestion at intermediate career stages is not a transient imbalance caused by short-term changes in degree output but instead reflects a persistent mismatch between the rate at which new candidates enter the hiring pool and the rate at which permanent academic positions become available. \\ 
 
 \noindent A central mechanism underlying these dynamics is vacancy--limited hiring. When faculty exits occur slowly, the number of available positions remains effectively fixed over long time horizons. Under these conditions, increases in candidate supply cannot be absorbed through faculty expansion and instead propagate backward through the system. The postdoctoral stage acts as the primary buffer for this imbalance, leading to sustained growth in postdoctoral populations. The persistence of this pattern across degree scenarios shows that reducing degree production alone does not resolve congestion when hiring capacity is constrained. The separation between the behavior of postdoctoral and faculty populations highlights the distinct roles of degree production and faculty turnover. Changes in degree output strongly affect the size of the postdoctoral pool and the intensity of competition for faculty positions, while faculty population levels remain largely unchanged under vacancy--limited hiring. This decoupling implies that faculty employment outcomes are insensitive to upstream growth unless turnover rates or hiring capacity are altered. As a result, policies or planning decisions based solely on degree counts risk misinterpreting the sources of workforce congestion. \\
 
 \noindent The analysis of competition intensity further clarifies how structural constraints shape career progression. As candidate supply increases relative to the number of available openings, competition intensifies and remains elevated over time. This competitive pressure is not alleviated through increased hiring under vacancy--limited dynamics, but instead becomes concentrated at the postdoctoral stage. The dominance of postdoctoral positions as the primary source of faculty hires reflects this accumulation, indicating that postdoctoral training has become an entrenched intermediate stage rather than a short transitional phase. Comparison with unconstrained hiring dynamics underscores the importance of explicitly representing hiring mechanisms. When hiring is not limited by vacancies, increases in candidate supply can be partially absorbed through faculty expansion, leading to lower sustained competition. In contrast, vacancy--limited hiring prevents such adjustment and locks excess supply into temporary positions. This contrast shows that assumptions about hiring constraints fundamentally alter system behavior and must be made explicit when interpreting observed workforce trends. Long--run outcomes are shown to be robust across a wide range of structural parameters. \\
 
 \noindent Sensitivity analyses reveal smooth and gradual dependence of congestion on faculty exit rates and capacity scales, with no abrupt transitions that eliminate accumulation. Lower exit rates consistently produce higher long--run postdoctoral to faculty ratios, while increases in capacity scale moderate but do not remove congestion. These patterns indicate that postdoctoral accumulation is a generic outcome of slow faculty turnover rather than a consequence of finely tuned parameter choices. The timing of congestion further emphasizes the long memory of the academic workforce. High postdoctoral to faculty ratios are reached relatively early under most parameter combinations and persist thereafter. Once established, these conditions are difficult to reverse because faculty turnover operates on much longer time scales than degree production. This temporal asymmetry explains why short-term adjustments to degree output have limited impact on long--run workforce structure. Several limitations of the modeling framework should be acknowledged. The model operates at an aggregate national level and does not distinguish between institution types or heterogeneous career pathways. Postdoctoral positions are treated as a single category, and faculty exits are modeled using constant probabilities rather than cohort specific processes. \\
 
 \noindent Overall, the results show that degree counts alone provide an incomplete representation of academic workforce dynamics. Persistent congestion can arise even under stable or declining degree production when hiring capacity is constrained by slow faculty turnover. Explicit representation of vacancy--limited hiring is therefore essential for understanding long--run career outcomes in the mathematical sciences and in other fields characterized by extended training pipelines and limited permanent positions. Future work could incorporate independent workforce datasets to assess order of magnitude consistency of postdoctoral and faculty stocks; however, such validation is beyond the scope of the present structural analysis.

\section{Conclusion}\label{sec5}

\noindent This study developed a discrete--time stock–flow model to examine long--run academic workforce dynamics in the mathematical sciences. The framework provides a structured representation of how observed degree production, faculty turnover, and hiring capacity jointly shape postdoctoral and faculty trajectories when advancement to permanent positions is constrained by vacancies. By explicitly distinguishing between observed degree flows and latent workforce stocks, the model allows downstream dynamics to be analyzed without inferring workforce outcomes directly from degree counts alone. The obtained results show that sustained accumulation at the postdoctoral stage can arise as a structural consequence of slow faculty turnover, independent of short-term changes in degree production. Under vacancy--limited hiring, increases in candidate supply do not translate into proportional faculty growth but instead propagate backward through the pipeline, leading to elevated postdoctoral stocks and intensified competition for faculty positions. Sensitivity analyses further show that adjustments to degree production alone have limited influence on long--run workforce structure unless accompanied by changes in faculty exit rates or hiring capacity. These findings highlight the importance of explicitly modeling hiring constraints when interpreting academic workforce trends or evaluating policy interventions. \\

\noindent While these results clarify how structural constraints shape long--run academic workforce dynamics, some limitations of the present analysis should be acknowledged. First, the model is calibrated using externally reported statistics and literature--based assumptions rather than estimated by fitting to workforce time series. As a result, the analysis is intended as a data--anchored, mechanistic investigation of structural dynamics rather than as a predictive or empirically validated forecasting model. Second, upstream undergraduate and graduate populations are reconstructed from observed degree data and therefore act as exogenous inputs to the postdoctoral and faculty system. This design precludes standard out-of-sample validation and implies that degree-series consistency checks are in-sample by construction. Third, while the sensitivity analysis identifies dominant parameters governing congestion and faculty outcomes, it does not fully capture higher-order parameter interactions, which may warrant future investigation using variance-based methods. Finally, the functional forms governing competition and hiring allocation are chosen for mathematical transparency and interpretability rather than empirical calibration, and alternative formulations could be explored to assess qualitative robustness. Despite these limitations, the framework introduced here provides a flexible approach for examining how structural constraints shape academic workforce dynamics in the mathematical sciences. Although the analysis focuses on a single discipline, the modeling strategy is applicable to other academic fields characterized by extended training pathways and limited permanent positions. Therefore, future work may incorporate independent workforce data for validation, explore alternative competition and hiring mechanisms, and extend the framework to account for institutional heterogeneity and cross-sector career pathways. Nonetheless, the central conclusion remains that long--run workforce outcomes are governed primarily by structural features of faculty turnover and hiring capacity rather than by degree production alone.

\subsection*{Acknowledgments}
  \noindent This work was conceived at the American Institute of Mathematics (AIM) workshop ``MetaMath: Modeling the mathematical sciences community using mathematics, statistics, and data science,'' held December 8--12, 2025 in Pasadena, California. The authors thank AIM for its hospitality and the workshop participants for valuable discussions. AIM is supported by the National Science Foundation.

 \subsection*{Funding}
   \noindent No funding was received for conducting this study.

   \subsection*{Competing Interests}
   \noindent The authors declare no competing interests relevant to the content of this article.

 \subsection*{Data Availability}
  \noindent The degree completion data analyzed in this study are publicly available from the National Center for Education Statistics (NCES), U.S. Department of Education, at \url{https://nces.ed.gov/programs/digest/d22/tables/dt22_325.65.asp}.

  \subsection*{Code Availability}
  \noindent The analysis code is available at \url{https://github.com/Oluwatosin-Babasola/MetaMath}.

  \subsection*{Disclaimer}  
  \noindent The views expressed in this article are those of the author(s) and do not necessarily reflect the official policy or position of the United States Air Force Academy, the Air Force, the Department of Defense, or the U.S. Government.

 \subsection*{Author Contributions}
 \noindent Oluwatosin Babasola: Writing original draft, Methodology, Conceptualization, Data curation, Visualization, Formal analysis; Olayemi Adeyemi: Validation, Writing review and editing; Ron Buckmire: Conceptualization, Validation, Writing review and editing, Data curation; Daozhou Gao: Conceptualization, Data curation, Validation, Writing review and editing; Maila Hallare: Data curation, Conceptualization, Writing review and editing, Validation; Olaniyi Iyiola: Conceptualization, Writing review and editing, Data curation, Validation; Deanna Needell: Conceptualization, Data curation, Writing review and editing; Chad M. Topazh: Conceptualization, Validation, Writing review and editing, Data curation; Andrés R. Vindas-Meléndez: Conceptualization, Validation, Writing review and editing, , Data curation. All authors have read and approved the final
version of the manuscript.

\bibliography{ref}
\end{document}